\documentclass[11pt]{article}
\usepackage[a4paper]{geometry}
\usepackage{amsthm}
\usepackage{amsmath}
\usepackage{amssymb}
\usepackage{amsfonts}
\usepackage{graphicx}
\usepackage{color}
\usepackage[bf,SL,BF]{subfigure}
\usepackage{url}
\usepackage{epstopdf}
\usepackage{pdfsync}
\usepackage[colorlinks]{hyperref}
\usepackage[all]{xypic}
\usepackage{comment}
\usepackage[colorinlistoftodos]{todonotes}
\usepackage{marginnote}
\usepackage{sistyle}
\SIthousandsep{,}

\graphicspath{ {images/} }

\hypersetup{
    linkcolor=blue,
}

\newlength{\fixboxwidth}
\usepackage{epsfig,amsbsy,graphicx,multirow}

\usepackage{ algorithm, algorithmic}
\renewcommand{\algorithmiccomment}[1]{\bgroup\hfill//~#1\egroup}

\usepackage[all]{xy}
\usepackage{accents}
\newcommand{\ubar}[1]{\underaccent{\bar}{#1}}

 \numberwithin{equation}{section}

\def\V{\Psi}

\def\P{\mathbb{P}}
\def\R{\mathbb{R}}

\def\cN{\mathcal{N}}

\def\E{\mathbb{E}}

\def\I{\mathcal{I}}
\def\J{\mathcal{J}}
\def\L{\mathcal{L}}

\def\W{\mathfrak{X}}

\def\H{\mathcal{H}}

\def\O{\mathcal{O}}
\def\N{\mathcal{N}}

\def\<{\big\langle}
\def\>{\big\rangle}
\def\Img{\operatorname{Im}}
\def\Ker{\operatorname{Ker}}
\def\Cond{\operatorname{Cond}}

\def\diiv{\operatorname{div}}

\def\Var{\operatorname{Var}}

\def\Cov{\operatorname{Cov}}
\def\Card{\operatorname{Card}}

\def\dim{{\operatorname{dim}}}

\def\Span{\operatorname{span}}

\DeclareMathOperator*{\argmin}{arg\,min}
\DeclareMathOperator*{\arginf}{arg\,inf}

\definecolor{red}{rgb}{0.9, 0, 0}

\newtheorem{Theorem}{Theorem}[section]

\newtheorem{Lemma}[Theorem]{Lemma}
\newtheorem{Corollary}[Theorem]{Corollary}
\newtheorem{Remark}[Theorem]{Remark}
\newtheorem{Example}[Theorem]{Example}

\newtheorem{Problem}{Problem}

\begin{document}
\title{De-noising by thresholding operator adapted wavelets}

\date{\today}

\author{Gene Ryan Yoo \thanks{California Institute of Technology, 1200 E California Blvd, MC 253-37, Pasadena, CA 91125, USA, gyoo@caltech.edu} and Houman Owhadi \thanks{California Institute of Technology, 1200 E California Blvd, MC 9-49, Pasadena, CA 91125, USA, owhadi@caltech.edu}}

\maketitle

\begin{abstract}
Donoho and Johnstone \cite{donoho1998minimax} proposed a method from reconstructing an unknown smooth function $u$ from noisy data $u+\zeta$ by translating the empirical wavelet coefficients of $u+\zeta$ towards zero. We consider the situation where the prior information on the unknown function $u$ may not be the regularity of $u$ but that of $ \L u$ where $\L$ is a linear operator (such as a PDE or a graph Laplacian).
We show that the approximation  of $u$ obtained  by thresholding the gamblet (operator adapted wavelet) coefficients of
$u+\zeta$ is near minimax optimal (up to a multiplicative constant), and with high probability, its  energy norm (defined by the operator) is  bounded
by that of $u$ up to a constant depending on the amplitude of the noise. Since gamblets can be computed in $\mathcal{O}(N \operatorname{polylog} N)$ complexity and are localized both in space and eigenspace, the proposed method is of near-linear complexity and generalizable to non-homogeneous noise.
\end{abstract}




\section{Introduction}
\cite{Donoho1990minimax, donoho1995noising, donoho1998minimax} addressed the problem of recovering of a smooth  signal  from noisy observations
  by soft-thresholding empirical wavelet coefficients \cite{donoho1995noising}.  More recently, \cite{Minimaxlinear2017} considered the recovery of $x \in X$ based on the observation of $Tx + \zeta$, where $\zeta_i$ is i.i.d. $\N(0, \sigma^2)$ and $T$ is a compact linear operator between Hilbert spaces $X$ and $Y$ with the prior that $x$ lies in an ellipsoid defined by the eigenvectors of $T^* T$.
    \cite{Minimaxlinear2017} showed that
     thresholding the coefficients of the corrupted signal $Tx + \zeta$ in the basis formed by
     the SVD of $T$ (which can be computed in   $\O(N^3)$ complexity) approached the minimax recovery to a fixed multiplicative constant.

In this paper we are interested in the fast recovery of a signal $u$ based on noisy observations $u+\zeta$ and a bound the regularity of $\L u$ where $\L$ is a linear operator. Our main motivation is to approximate solutions of PDEs or Graph
Laplacians based on their noisy observations.

Our first setting will be that of a symmetric positive linear bijection mapping $\H^s_0(\Omega)$ to $\H^{-s}(\Omega)$, i.e.
   \begin{equation}\label{eqkugdeuydgd}
\L\,:\,\H^s_0(\Omega)\rightarrow \H^{-s}(\Omega)\,,
\end{equation}
where $s\in \mathbb{N}^*$ and $\Omega$ be a regular bounded domain of $\R^d$ ($d\in \mathbb{N}$).
We also  assume $\L$ to be local, i.e.
$\int_{\Omega} u \L v=0$ for $u,v\in \H^s_0(\Omega)$ with disjoint supports (this assumption is used to achieve near-linear complexity in the recovery).  Let $\|\cdot\|$ be the energy-norm defined by
\begin{equation}
\|u\|^2:=\int_{\Omega} u\L u\,,
\end{equation}
and write
\begin{equation}
\<u,v\>:=\int_{\Omega} u\L v\,,
\end{equation}
for the associated scalar product.

Let
\begin{equation}\label{eqjkjhbejhbdd}
\zeta\sim \cN(0,\sigma^2 \delta(x-y))\,,
\end{equation}
be white noise defined as a centered Gaussian process on $\Omega$ with covariance function $\sigma^2 \delta(x-y)$.
Consider the following problem.
\begin{Problem}\label{pb1}
Let $u$ be an unknown element of $\H^s_0(\Omega)$. Given the noisy observation $\eta=u+\zeta$ and a prior bound on $\| \L u\|_{L^2}$, find an approximation of $u$ that is as accurate as possible in the energy norm $\|\cdot\|$.
\end{Problem}
\begin{Example}\label{egprotoa2}
As a running illustrative example we will consider the case where $s=1$ and $\L$ is the differential operator $-\diiv \big(a(x)  \nabla \cdot\big)$ where the conductivity $a$ is a uniformly elliptic symmetric $d\times d$ matrix with entries in $L^\infty(\Omega)$.
 We define $\lambda_{\min}(a)$ as the largest constant and $\lambda_{\max}(a)$ as the smallest constant such that for all $x\in \Omega$ and $l\in \R^d$,
\begin{equation}
\lambda_{\min}(a) |l|^2 \leq l^T a(x) l \leq \lambda_{\max}(a)|l|^2.
\end{equation}
Problem \ref{pb1} then corresponds to the problem of recovering the solution of the PDE
\begin{equation}\label{eqn:scalarprotoa2}
\begin{cases}
    -\diiv \big(a(x)  \nabla u(x)\big)=f(x) &  x \in \Omega; \\
    u=0 & \text{on}\quad \partial \Omega,
\end{cases}
\end{equation}
from its noisy  observation $\eta=u+\zeta$ and a prior bound on $\|f\|_{L^2(\Omega)}$.
\end{Example}

To solve  Problem \ref{pb1} we will decompose $\eta$ over wavelets adapted to the operator $\L$ and filter its wavelet coefficients. In Section \ref{secgambl}, we will first summarize (see \cite{OwhadiMultigrid:2015,OwhadiScovel:2017, OwZh:2016, OwhScobook2018, SchaeferSullivanOwhadi17} for details) the main properties of these operator adapted wavelets, named \emph{gamblets} in reference to their game theoretic interpretation \cite{OwhadiMultigrid:2015, OwhadiScovel:2017,  OwhScobook2018}.
Given that \emph{gamblets}  can also be interpreted in the frameworks of Information Based Complexity \cite{Woniakowski1986},
 Bayesian Numerical Analysis \cite{Diaconis:1988},
 Optimal Recovery \cite{micchelli1977survey}, and Probabilistic Numerics \cite{Hennig2015},
the
 results of this paper suggest that probabilistic numerical methods \cite{Hennig2015, ChkrebtiiCampbell2015, Owhadi:2014,  Briol2015,  cockayne2016probabilistic,  raissi2017inferring, Cockayne2017} may not only lead to efficient quadrature rules \cite{schober2014nips}, seamless integration of model uncertainty with numerical errors \cite{oates2017bayesian}, fast solvers \cite{SchaeferSullivanOwhadi17}, and
 optimal  pipelines of computation \cite{Cockayne2017}, they may also lead to near-optimal methods for the de-noising of solutions of differential equations (see Remark \ref{rmkhghvvgyg}).

\section{Gamblets}\label{secgambl}
\subsection{Hierarchy of measurement functions}
Let $q\in \mathbb{N}^*$ (used to represent a number of scales). Let $(\I^{(k)})_{1\leq k \leq q}$ be a hierarchy of labels and let $\{\phi_i^{(k)}|k\in \{1,\ldots,q\},\, i\in \I^{(k)}\}$ be a nested hierarchy of elements of $\H^{-s}(\Omega)$ such that
\begin{equation} \label{nesting relation}
\phi_i^{(k)}=\sum_{j\in \I^{(k+1)}}\pi_{i,j}^{(k,k+1)}\phi_j^{(k+1)}\,,
\end{equation}
for $i\in \I^{(k)}$, $k\in \{1,\ldots,q-1\}$, where $\pi^{(k,k+1)}$ is an $\I^{(k)}\times \I^{(k+1)}$ matrix. Assume that the $(\phi_i^{(q)})_{i\in \I^{(q)}}$ are linearly independent and (writing $\pi^{(k+1,k)}$ for the transpose of $\pi^{(k,k+1)}$ and
$I^{(k)}$ for the $\I^{(k)}\times \I^{(k)}$ matrix)
\begin{equation}
\pi^{(k,k+1)}\pi^{(k+1,k)}=I^{(k)}\,.
\end{equation}
\subsection{Hierarchy of operator adapted pre-wavelets}
Let $(\psi_i^{(k)})_{i\in \I^{(k)}}$ be the hierarchy of optimal recovery splines associated with $(\phi_i^{(k)})_{i\in \I^{(k)}}$, i.e. for $k\in \{1,\ldots,q\}$ and $i\in \I^{(k)}$,
\begin{equation}
\psi_i^{(k)}=\sum_{j\in \I^{(k)}} A^{(k)}_{i,j} \L^{-1}\phi_j^{(k)}\,,
\end{equation}
where
\begin{equation}
A^{(k)}:=(\Theta^{(k)})^{-1}\,,
\end{equation}
and $\Theta^{(k)}$ is the $\I^{(k)}\times \I^{(k)}$ symmetric positive definite Gramian matrix with entries (writing $[\phi,v]$ for the duality pairing between $\phi\in \H^{-s}(\Omega)$ and $v\in \H^s_0(\Omega)$)
\begin{equation}
\Theta^{(k)}_{i,j}=[\phi_i^{(k)},\L^{-1}\phi_j^{(k)}]\,.
\end{equation}
Note that
\begin{equation}
    A^{(k)}_{i,j} = \<\psi_i^{(k)}, \psi_j^{(k)}\>.
\end{equation}
Writing
$\Phi^{(k)}:=\Span\{\phi_i^{(k)}\mid i\in \I^{(k)}\}$ and $\V^{(k)}:=\Span\{\psi_i^{(k)}\mid i\in \I^{(k)}\}$,
$\Phi^{(k)}\subset \Phi^{(k+1)}$ and $\Psi^{(k)}=\L^{-1} \Phi^{(k)}$ imply  $\Psi^{(k)}\subset \Psi^{(k+1)}$.
The $(\phi_i^{(k)})_{i\in \I^{(k)}}$ and $(\psi_i^{(k)})_{i\in \I^{(k)}}$ form a bi-orthogonal system in the sense that
\begin{equation}
[\phi_i^{(k)},\psi_j^{(k)}]=\delta_{i,j}\text{ for }i,j\in \I^{(k)}\,,
\end{equation}
and the $\<\cdot,\cdot\>$-orthogonal projection of $u\in \H^s_0(\Omega)$ on $\Psi^{(k)}$ is
\begin{equation}\label{eqkjhekjdhdjnd}
u^{(k)}:= \sum_{i \in \I^{(k)}}[\phi_i^{(k)},u]\psi_i^{(k)}\,.
\end{equation}
\subsection{Operator adapted wavelets}
Let $(\J^{(k)})_{2\leq k\leq q}$ be hierarchy of labels such that (writing $|\J^{(k)}|$ for the cardinal of $\J^{(k)}$)
\begin{equation}\label{eqkejhdkdhjd}
|\J^{(k)}|=|\I^{(k)}|-|\I^{(k-1)}|\,.
\end{equation}
For $k\in \{2,\ldots,q\}$, let $W^{(k)}$ be a $\J^{(k)}\times \I^{(k)}$ matrix such that (writing $W^{(k),T}$ for the transpose of $W^{(k)}$)
\begin{equation}
\Ker(\pi^{(k-1,k)})=\Img(W^{(k),T})\,.
\end{equation}
For $k\in \{2,\ldots,q\}$ and $i\in \J^{(k)}$ define
\begin{equation}
\chi_i^{(k)}:=\sum_{j\in \I^{(k)}} W^{(k)}_{i,j} \psi_j^{(k)}\,,
\end{equation}
and write
$
\W^{(k)}:=\Span\{\chi_i^{(k)}\mid i\in \J^{(k)}\}\,.
$
Then $\W^{(k)}$ is the $\<\cdot,\cdot\>$-orthogonal complement of $\V^{(k-1)}$ in $\V^{(k)}$, i.e.
$
\V^{(k)}=\V^{(k-1)}\oplus \W^{(k)}\,,
$
and
\begin{equation}
\V^{(q)}=\V^{(1)}\oplus \W^{(2)}\oplus \cdots \oplus \W^{(q)}\,.
\end{equation}
For $k\in \{2,\ldots,q\}$ write
\begin{equation}
B^{(k)}:= W^{(k)}A^{(k)}W^{(k),T}\text{ and }N^{(k)}:=A^{(k)} W^{(k),T} B^{(k),-1}\,,
\end{equation}
and for $i\in \J^{(k)}$,
\begin{equation}
\phi_i^{(k),\chi}:=\sum_{j\in \I^{(k)}} N^{(k),T}_{i,j}\phi_j^{(k)}\,.
\end{equation}
Then defining
$u^{(k)}$ as in \eqref{eqkjhekjdhdjnd}, it holds true that for $k\in \{2,\ldots,q\}$, $u^{(k)}-u^{(k-1)}$ is the $\<\cdot,\cdot\>$-orthogonal projection of $u$ on $\W^{(k)}$ and
\begin{equation}
u^{(k)}-u^{(k-1)}=\sum_{i\in \J^{(k)}} [\phi_i^{(k),\chi},u]\chi_i^{(k)}\,.
\end{equation}
To simplify notations write $\J^{(1)}:=\I^{(1)}$, $B^{(1)}:=A^{(1)}$, $N^{(1)}:=I^{(1)}$, $\phi_i^{(1),\chi}:=\phi_i^{(1)}$ for  $i\in \J^{(1)}$,
$\J:=\J^{(1)}\cup \cdots \cup \J^{(q)}$, $\chi_i:=\chi_i^{(k)}$ and $\phi_i^\chi:=\phi_i^{(k),\chi}$ for $i\in \J^{(k)}$ (note that the use of $\phi^{(k),\chi}_i$ in an equation will imply $i \in \J^{(k)}$).
Then the $\phi_i^\chi$ and $\chi_i$ form a bi-orthogonal system, i.e.
\begin{equation}
[\phi_i^\chi,\chi_j]=\delta_{i,j}\text{ for } i,j\in \J\,,
\end{equation}
and
\begin{equation}\label{eqkljedhediduhd}
u^{(q)}=\sum_{i\in \J}[\phi_i^\chi,u]\chi_i\,.
\end{equation}
Simplifying notations further, we will write $[\phi^\chi,u]$ for the $\J$ vector with entries $[\phi_i^\chi,u]$ and $\chi$ for the $\J$ vector with entries $\chi_i$ so that \eqref{eqkljedhediduhd} can be written
\begin{equation}\label{eqjkhdgiehjddf}
u^{(q)}=[\phi^\chi,u]\cdot \chi\,.
\end{equation}
Further, define the $\J$ by $\J$ block-diagonal matrix $B$ defined as $B_{i,j} = B^{(k)}_{i,j}$ if $i, j \in \J^{(k)}$ and $B_{i, j} = 0$ otherwise.  Note that it holds that $B_{i, j} = \< \chi_i, \chi_j \>$.
When $q=\infty$, and $\cup_{k=1}^\infty \Phi^{(k)}$ is dense in $\H^{-s}(\Omega)$ then (writing $\W^{(1)}:=\V^{(1)}$)
\begin{equation}
\H^s_0(\Omega)=\oplus_{k=1}^\infty \W^{(k)}\,,
\end{equation}
$u^{(q)}=u$ and \eqref{eqkljedhediduhd} is the corresponding multi-resolution decomposition of $u$. When $q<\infty$, $u^{(q)}$ is the projection of $u$ on $\oplus_{k=1}^q \W^{(k)}$ and \eqref{eqjkhdgiehjddf} is the corresponding multi-resolution decomposition.

\subsection{Pre-Haar Wavelet Measurement Functions} \label{pre-haar}

\begin{Example}\label{exprehaar}
Let  $\I^{(q)}$ be the finite set of $q$-tuples of the form $i=(i_1,\ldots,i_q)$.  For $1\leq k<r$ and an $r$-tuple of the form $i=(i_1,\ldots,i_r)$, write $i^{(k)}:=(i_1,\ldots,i_k)$.
For $1\leq k \leq q$ and $i=(i_1,\ldots,i_q)\in \I^{(q)}$,  write $\I^{(k)}:=\{i^{(k)}\,:\, i\in \I^{(q)}\}$.
Let $\delta, h\in (0,1)$.
   Let  $(\tau_i^{(k)})_{i\in \I^{(k)}}$ be  uniformly Lipschitz convex sets forming a nested partition of $\Omega$, i.e. such that
$
   \Omega=\cup_{i\in \I^{(k)}}\tau_i^{(k)},\quad k \in \{1,\ldots, q\}
$
    is a disjoint union except for the boundaries,  and
$
    \tau_i^{(k)}=\cup_{j\in \I^{(k+1)}: j^{(k)}=i}\tau_j^{(k+1)}, \quad  k \in \{1,\ldots, q-1\}$.
      Assume that each $\tau_i^{(k)}$,  contains a ball  of radius $\delta h^k$,
 and is contained in the  ball of radius $ \delta^{-1}h^k$. Writing $|\tau^{(k)}_i|$ for the volume of $\tau^{(k)}_i$, take
 \begin{equation}
\phi_i^{(k)}:=1_{\tau^{(k)}_i} |\tau^{(k)}_i|^{-\frac{1}{2}}\,.
\end{equation}
The nesting relation \eqref{nesting relation} is then satisfied with
$\pi^{(k,k+1)}_{i,j}:=
|\tau^{(k+1)}_j|^{\frac{1}{2}}|\tau^{(k)}_i|^{-\frac{1}{2}}$ for $j^{(k)}=i$ and $\pi^{(k,k+1)}_{i,j}:=0$ otherwise.

For $k\in \{2,\ldots,q\}$ let $\J^{(k)}$ be a finite set of $k$-tuples of the form $j=(j_1,\ldots,j_k)$
such that $\{j^{(k-1)}\mid j\in \J^{(k)}\}=\I^{(k-1)}$ and for $i\in \I^{(k-1)}$, $\Card\{ j\in \J^{(k)}\mid j^{(k-1)}=i\}=\Card\{ s\in \I^{(k)}\mid s^{(k-1)}=i\}-1$. Note that the cardinalities of these sets satisfy \eqref{eqkejhdkdhjd}.

Write $J^{(k)}$ for the $\J^{(k)}\times \J^{(k)}$ identity matrix.
For $k=2,\ldots,q$ let $W^{(k)}$ be a $\J^{(k)}\times \I^{(k)}$ matrix such that  $\Img(W^{(k),T})=\Ker(\pi^{(k-1,k)})$, $W^{(k)}(W^{(k)})^T=J^{(k)}$ and $W^{(k)}_{i,j}=0$ for $i^{(k-1)}\not=j^{(k-1)}$.
\end{Example}

\begin{Theorem}
\label{themoptimaldecomposition}
Under Example \ref{exprehaar}, it holds true that
\begin{enumerate}
\item For $k\in \{1,\ldots,q\}$ and $u\in \L^{-1}L^2(\Omega)$,
\begin{equation}
\| u - u^{(k)}\|   \leq C h^{ks}\|\L u\|_{L^2(\Omega)} \,.
\end{equation}
\item  Writing $\Cond(M)$ for the condition number of a matrix $M$  we have
 for $k\in \{1, \cdots, q\}$,
\begin{equation}\label{eqcond}
C^{-1}h^{-2(k-1)s}J^{(k)}\leq B^{(k)} \leq Ch^{-2ks}J^{(k)}\text{ and }\Cond(B^{(k)})\leq Ch^{-2s}\,.
\end{equation}
\item For $i\in \I^{(k)}$ and $x_i^{(k)}\in \tau_i^{(k)}$
\begin{equation}\label{eqatycc2}
\|\psi_i\|_{\H^s(\Omega\setminus B(x_i^{(k)},nh))} \leq  C h^{-s}e^{-n/C}
\end{equation}
\item The wavelets $\psi_i^{(k)}, \chi_i^{(k)}$ and stiffness matrices $A^{(k)}, B^{(k)}$ can be computed can be computed to precision $\epsilon$ (in $\|\cdot\|$-energy norm for elements of $\H^s_0(\Omega)$ and in Frobenius norm for matrices)
 in $\operatorname{O}(N \log^{3d} \frac{N}{\epsilon})$ complexity.
\end{enumerate}
Furthermore the constant $C$ depends only on
$\delta, \Omega, d, s$,
\begin{equation}
\label{eqkdjhekdfhfu}
 \|\L\|:=\sup_{u\in \H^s_0(\Omega)} \frac{\|\L u\|_{\H^{-s}(\Omega)}}{ \|u\|_{\H^s_0(\Omega)}}\text{ and }
\|\L^{-1}\|:=\sup_{u\in \H^s_0(\Omega)} \frac{ \|u\|_{\H^s_0(\Omega)}}{\|\L u\|_{\H^{-s}(\Omega)}}\,.
\end{equation}
\end{Theorem}

\begin{Remark}
The wavelets $\psi_i^{(k)}, \chi_i^{(k)}$ and stiffness matrices $A^{(k)}, B^{(k)}$ can also be computed in $\operatorname{O}(N \log^{2} N \log^{2d} \frac{N}{\epsilon})$ complexity using the incomplete Cholesky factorization approach of \cite{SchaeferSullivanOwhadi17}.
\end{Remark}

\section{On $\R^N$.} \label{gamblets R^N}
Gamblets have a natural generalization in which $\H^s_0(\Omega)$ and $\H^{-s}(\Omega)$ are replaced by $\R^N$ (or a finite-dimensional space) and $\L$ is replaced by a $N\times N$ symmetric positive definite matrix $A$ \cite{OwhadiScovel:2017,  OwhScobook2018}.
This generalization is relevant to practical applications requiring the prior numerical discretization of the continuous operator $\L$.
In these applications (1) $\H^s_0(\Omega)$ is replaced by  the linear space spanned by the finite-elements $\tilde{\psi}_i$ (e.g.~piecewise linear or bi-linear tent functions on a fine mesh/grid in the setting of Example \ref{egprotoa2}) used  to discretize the operator $\L$ (2) $\I^{(q)}$ is used to label the elements $\tilde{\psi}_i$ (3) we have $\psi_i^{(q)}=\tilde{\psi}_i$ for $i\in \I^{(q)}$.
 Algorithms \ref{alggtphiior} and  \ref{algts} summarize
the discrete gamblet-transform \cite{OwhadiMultigrid:2015,OwhadiScovel:2017,  OwhScobook2018} and the discrete Gamblet solve of the linear system $\L u =f$.
The near-linear complexities mentioned in Theorem \ref{themoptimaldecomposition} are based on the near-sparsity of the interpolation matrices $R^{(k-1,k)}$ and the well-conditioning and near-sparsity of the $B^{(k)}$, and
achieved by localizing the computation of gamblets, the inversion of $B^{(k)}$ and truncating the $A^{(k)}$ (see \cite{OwhadiMultigrid:2015,OwhadiScovel:2017,  OwhScobook2018} for details).

\begin{algorithm}[h]
\caption{The Gamblet Transform.}\label{alggtphiior}
\begin{algorithmic}[1]
\STATE\label{step3g} $\psi^{(q)}_i= \tilde{\psi}_i$
\STATE\label{step5g} $A^{(q)}_{i,j}= \< \psi_i^{(q)}, \psi_j^{(q)}\>$
\FOR{$k=q$ to $2$}
\STATE\label{step7g} $B^{(k)}= W^{(k)}A^{(k)}W^{(k),T}$
\STATE\label{step9g}  $\chi^{(k)}_i=\sum_{j \in \I^{(k)}} W_{i,j}^{(k)} \psi_j^{(k)}$
\STATE\label{step12g} $ R^{(k-1,k)}=\pi^{(k-1,k)} (I^{(k)}-A^{(k)} W^{(k),T}B^{(k),-1}
W^{(k)})$
\STATE\label{step13g} $A^{(k-1)}= R^{(k-1,k)}A^{(k)}R^{(k,k-1)}$
\STATE\label{step14g} $\psi^{(k-1)}_i=\sum_{j \in  \I^{(k)}} R_{i,j}^{(k-1,k)} \psi_j^{(k)}$
\ENDFOR
\end{algorithmic}
\end{algorithm}

\begin{algorithm}[h!]
\caption{The Gamblet Solve.}\label{algts}
\begin{algorithmic}[1]
\STATE\label{step4g} $f^{(q)}_i=\int_{\Omega}f\psi_i^{(q)}$
\FOR{$k=q$ to $2$}
\STATE\label{step8g} $w^{(k)}=B^{(k),-1} W^{(k)} f^{(k)}$
\STATE\label{step10g} $v^{(k)}=\sum_{i\in \J^{(k)}}w^{(k)}_i \chi^{(k)}_i$ \COMMENT{$v^{(k)}=u^{(k)}-u^{(k-1)}$}
\STATE\label{step15g} $f^{(k-1)}=R^{(k-1,k)} f^{(k)}$
\ENDFOR
\STATE\label{step16g} $ w^{(1)}=A^{(1),-1}f^{(1)}$
\STATE\label{step17g} $v^{(1)}=\sum_{i \in \I^{(1)}} w^{(1)}_i \psi^{(1)}_i$ \COMMENT{$v^{(1)}=u^{(1)}$}
\STATE\label{step18g} $u=v^{(1)}+v^{(2)}+\cdots+v^{(q)}$
\end{algorithmic}
\end{algorithm}

\section{De-noising by filtering gamblet coefficients}\label{pde sol recovery}
In this section we will show that filtering the gamblet coefficients of $\eta$ in Problem \ref{pb1} produces an approximation of $u$ that is minimax optimal up to a multiplicative constant.

\subsection{Near minimax recovery} \label{Near Minimax Sec}
In the continuous setting, the near optimal recovery of $u$ from $\eta$ from filtering gamblet coefficients requires the finest wavelet resolution $h^q$ to be small enough compared to the level of noise, i.e. $(\frac{\sigma}{M})^\frac{2}{4s+d}>    h^{q }$. If $(\frac{\sigma}{M})^\frac{2}{4s+d}<   h^{q }$ then the proposed estimator is near optimal only for the recovery of $u^{(q)}$. Since by Theorem \ref{themoptimaldecomposition},
$\| u - u^{(q)}\|   \leq C h^{qs}\|f\|_{L^2(\Omega)}$  this is not a restriction if $h^q$ is small enough.
We  will therefore state the near optimality of the recovery in the setting of  Problem \ref{discrete prob},
a discrete variant of Problem \ref{pb1}, in which one tries to recover $u\in \Psi^{(q)}$ given the observation
 $\eta=u+\zeta$ where (by \eqref{eqkjhekjdhdjnd})  $\zeta\in \Psi^{(q)}$ is the projection of the noise \eqref{eqjkjhbejhbdd} onto $\Psi^{(q)}$.

\begin{Problem} \label{discrete prob}
    Let $u$ be an unknown element of $\Psi^{(q)} \subset \H^s_0(\Omega)$ for $q < \infty$. Let $\zeta$ be a centered Gaussian vector in $\Psi^{(q)}$ such that $\E\big[ [\phi^{(q)}_i, \zeta] [\phi^{(q)}_j, \zeta] \big] = \sigma^2 \delta_{i,j}$.
     Given the  noisy observation $\eta=u+\zeta$ and a prior bound $M$ on $\| \L u\|_{L^2}$, find an approximation of $u$ in $\Psi^{(q)}$ that is as accurate as possible in the energy norm $\|\cdot\|$.
\end{Problem}
Let $\eta$ be as in  Problem \ref{discrete prob} and gamblets be defined as in Example \ref{exprehaar}.
For $l\in \{1,\ldots,q\}$ let
\begin{equation}\label{eqkjehgdhdgfh}
\eta^{(l)}:=\sum_{k=1}^{l} [\phi^{(k),\chi},\eta]\cdot \chi^{(k)}
\end{equation}
and $\eta^{(0)} = 0 \in \Psi^{(q)}$.  Let $M>0$ and write
\begin{equation}
V_M^{(q)} = \{ u\in \Psi^{(q)} \mid \|\L u\|_{L^2(\Omega)} \leq M\}\,.
\end{equation}
Assume that $\sigma>0$ and write
\begin{equation}\label{l^dag def}
    l^\dagger = \argmin_{l\in \{0,\dots,q\}} \beta_l,
\end{equation}
for
\begin{equation} \label{beta def}
    \beta_l = \begin{cases}
    h^{2s} M^2 & \text{if } l = 0\\
    \sigma^2 h^{-(2s+d)l} + h^{2s(l+1)}M^2 & \text{if } 1 \leq l \leq q - 1\\
    h^{-(2s+d)q} \sigma^2 & \text{if } l = q \, .
    \end{cases}
\end{equation}

The following theorem shows that $\eta^{(l^\dagger)}$ is a near minimax recovery of $u$, which, with probability close to $1$, is nearly as regular (in energy norm) as $u$.
\begin{Theorem} \label{main thm}
Suppose $v^\dagger(\eta) = \eta^{(l^\dagger)}$, then there exists a constant $C$  depending only on $h$, $s$, $\|\L\|$, $\|\L^{-1}\|$, $\Omega$, $d$, and $\delta$ such that
\begin{equation} \label{minimax}
    \sup_{u\in V_M^{(q)}} \E\big[\|u - v^\dagger(\eta)\|^2 \big] < C \inf_{v(\eta)} \sup_{u \in V_M^{(q)}} \E \big[\|u - v(\eta)\|^2\big]\,,
\end{equation}
where the infimum is taken over all measurable functions $v:\Psi^{(q)} \rightarrow \Psi^{(q)}$. Furthermore, if $l^\dagger \neq 0$, then with probability at least $1 - \varepsilon$,
    \begin{equation}
        \|\eta^{(l^\dagger)}\| \leq \|u\| + C \sqrt{\log\frac{1}{\varepsilon}} \sigma^\frac{2s+d}{4s+d} M^\frac{2s+2d}{4s+d}
    \end{equation}
\end{Theorem}
\begin{proof}
See Section \ref{proof of minimax}.
\end{proof}
Note that $l^\dagger = q$ occurs (approximately) when $q$ is such that $h^q > (\frac{\sigma}{M})^\frac{2}{4s+d}$, i.e. when
\begin{equation}
   \frac{\sigma}{M}<   h^{q \frac{4s+d}{2}}\,,
\end{equation}
and in this case $\eta^{(q)}$ is a near minimax optimal recovery of $u^{(q)}$. On the other extreme  $l^\dagger = 0$ occurs  (approximately) when $(\frac{\sigma}{M})^\frac{2}{4s+d} > h$, i.e. when
\begin{equation}
\frac{\sigma}{M}> h^{\frac{4s+d}{2}}
\end{equation}
 and in this case, the zero signal is a near optimal recovery.

\begin{Remark}\label{rmkhghvvgyg}
Write $\xi\sim \cN(0,\L^{-1})$ for the centered Gaussian field on $\H^s_0(\Omega)$ with covariance operator $\L^{-1}$ (i.e. $[\phi,\xi]\sim \cN(0, [\phi,\L^{-1}\phi])$ for $\phi\in \H^{-s}(\Omega)$. According to
\cite{OwhadiMultigrid:2015, OwhadiScovel:2017,  OwhScobook2018}, $\xi$ emerges as an optimal mixed strategy in an adversarial game opposing 2 players (in this game Player I selects $u\in H^s_0(\Omega)$ and player II must approximate $u$ with $v$ based on a finite number of linear measurements on $u$ and receives the loss $\|u-v\|/\|u\|$). Furthermore for $k\in \{1,\ldots,q\}$,
$\eta^{(k)}$ also admits the representation
\begin{equation}
\eta^{(k)}=\E\big[\xi\mid [\phi^{(k)},\xi]=[\phi^{(k)},\eta]\big]\,.
\end{equation}
Therefore the near-optimal recovery of $u$ can be obtained by conditioning the optimal mixed strategy $\xi$ with respect to linear measurements on $\eta$ made at level $k=l^\dagger$.
\end{Remark}

\subsection{Numerical illustrations}
\subsubsection{Example \ref{egprotoa2} with $d=1$}\label{minimax plots 1D}
Consider Example \ref{egprotoa2} with $d=1$.  Take $\Omega = [0, 1] \in \mathbb{R}$, $q = 10$ and $\phi^{(k)}_i = 1_{[\frac{i - 1}{2^k}, \frac{i}{2^k}]}$ for $1 \leq i \leq 2^k$.  Let $W^{(k)}$ be the $2^{k-1}$ by $2^k$ matrix with non-zero entries defined by $W_{i, 2i-1} = \frac{1}{\sqrt{2}}$ and $W_{i, 2i} = -\frac{1}{\sqrt{2}}$.

Let $\L := -\diiv (a  \nabla \cdot)$ with
\begin{equation}
    a(x) := \prod_{k = 1}^{10} (1 + 0.25 \cos(2^k x)).
\end{equation}

Select $f(x)$ at random using the uniform distribution on the unit $L^2(\Omega)$-sphere of $\Phi^{(q)}$.  Average errors are computed using  $3,000$  randomly (and independently) generated $f$.  Let $\zeta$  be white noise (as in \eqref{eqjkjhbejhbdd}) with $\sigma = 0.001$ and $\eta = u + \zeta$. See Fig.~\ref{plot a, f, u, eta, 1D} for numerical illustrations with a particular realization of $f$.

\begin{figure}[h] \label{minimax_1D}
    \centering
    \includegraphics[width=0.9\textwidth]{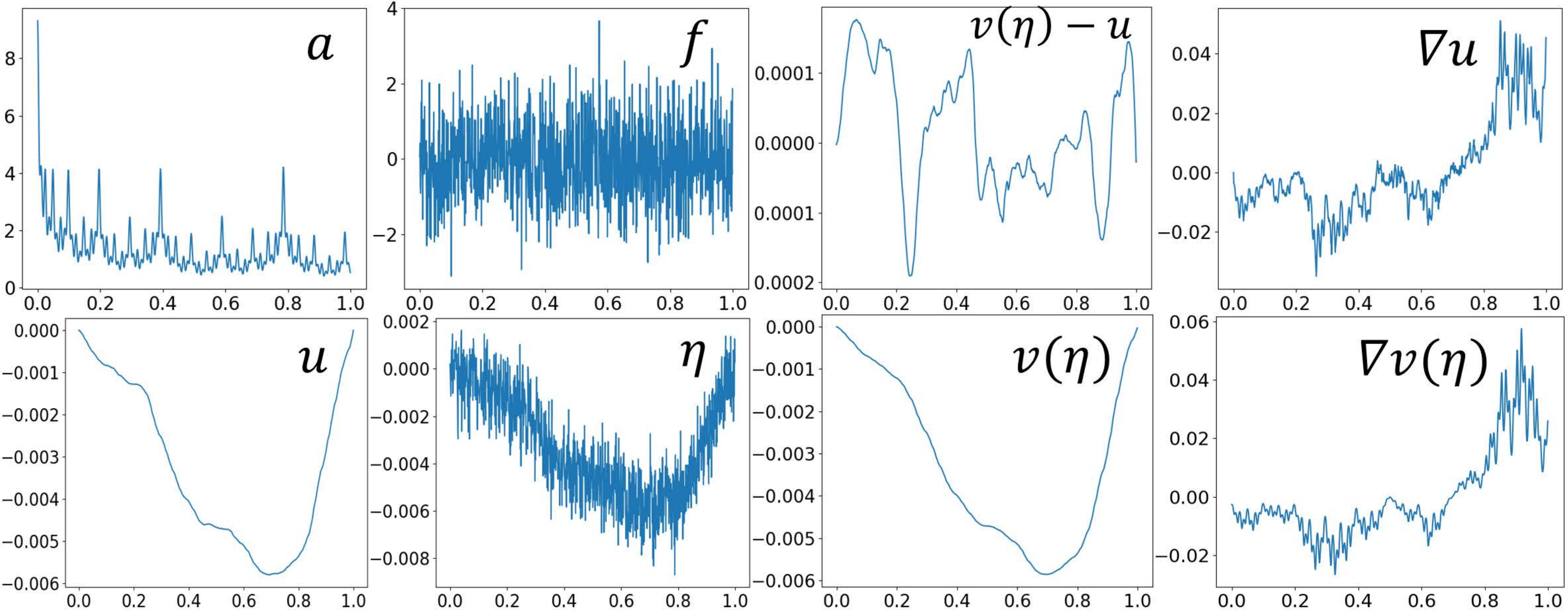}
    \caption{The plots of $a$, $f$, $u$, $\eta$, the near minimax recovery $v(\eta)=\eta^{(l^\dagger)}$, its error from $u$, and the derivatives of $u$ and $v(\eta)$}
    \label{plot a, f, u, eta, 1D}
\end{figure}

Next consider a case where $f$ is smooth, i.e. $f(x) = \frac{\sin(\pi x)}{x}$ on $x\in(0, 1]$ and $f(0) = \pi$. Let  $\zeta$ be white noise with standard deviation $\sigma = 0.01$.  See \ref{smooth plot a, f, u, eta, 1D} for the corresponding numerical illustrations.

\begin{figure}[h] \label{minimax_smooth_1D}
    \centering
    \includegraphics[width=0.9\textwidth]{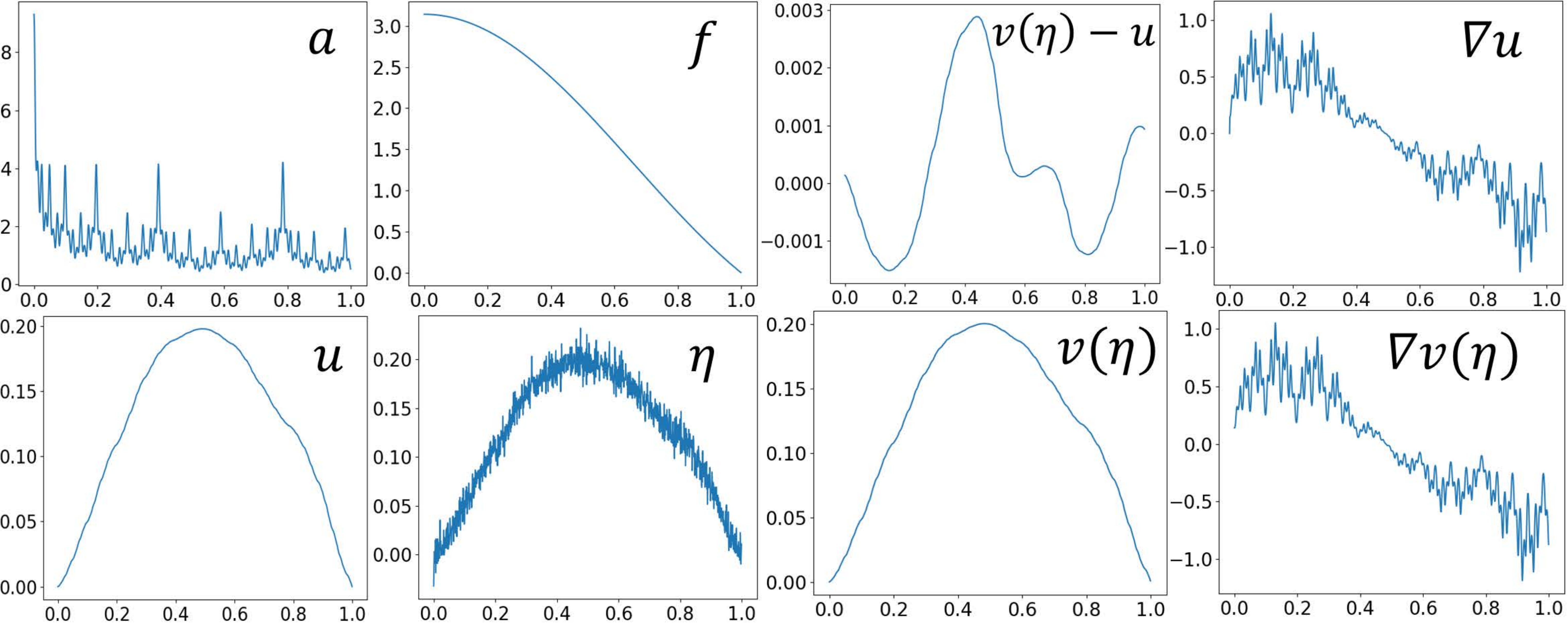}
    \caption{The plots of $a$, smooth $f$, $u$, $\eta$, the near minimax recovery $v(\eta)=\eta^{(l^\dagger)}$, its error from $u$, and the derivatives of $u$ and $v(\eta)$}
    \label{smooth plot a, f, u, eta, 1D}
\end{figure}
\begin{figure}[h]
    \centering
    \includegraphics[width=0.9\textwidth]{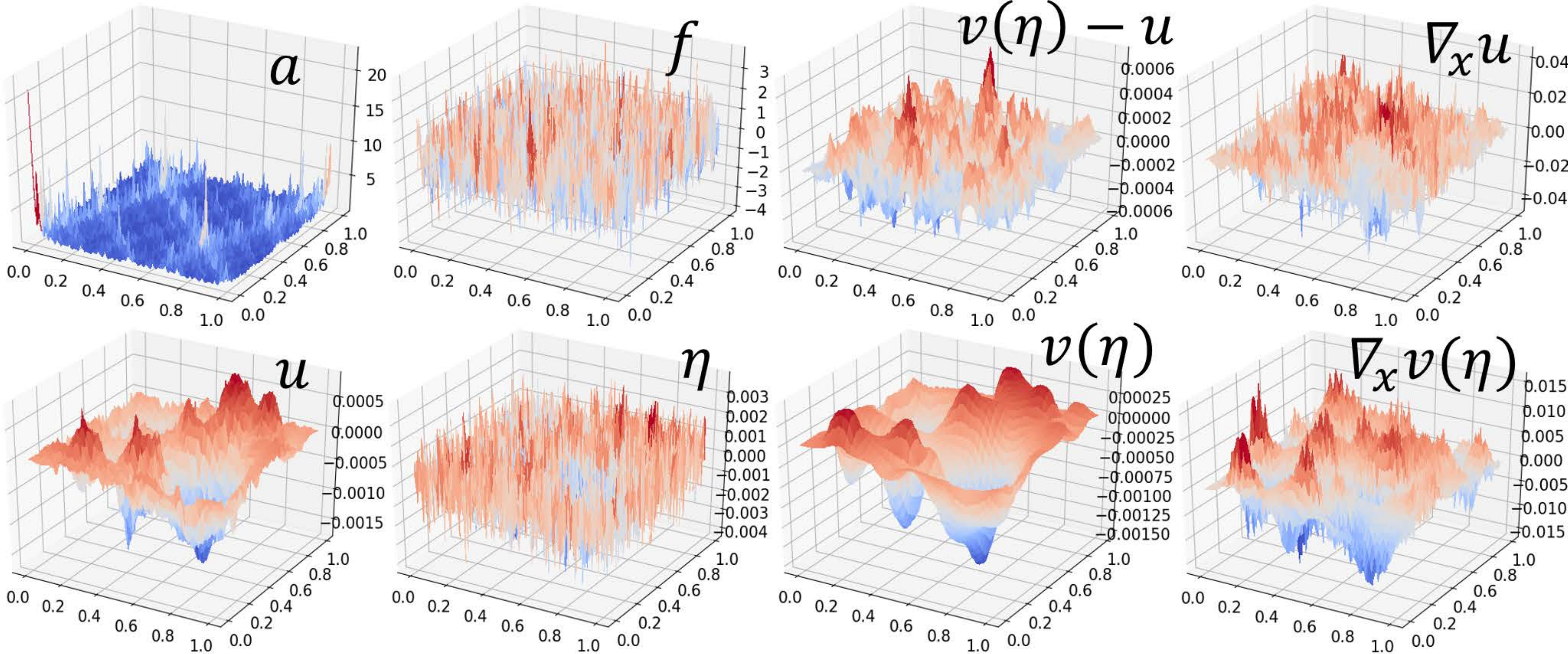}
    \caption{The plots of $a$, $f$, $u$, $\eta$, the near minimax recovery $v(\eta)=\eta^{(l^\dagger)}$, its error from $u$, and the gradient of $u$ and $v(\eta)$}
    \label{plot minimax a, f, u, eta 2D}
\end{figure}
\begin{figure}[h]
    \centering
    \includegraphics[width=0.9\textwidth]{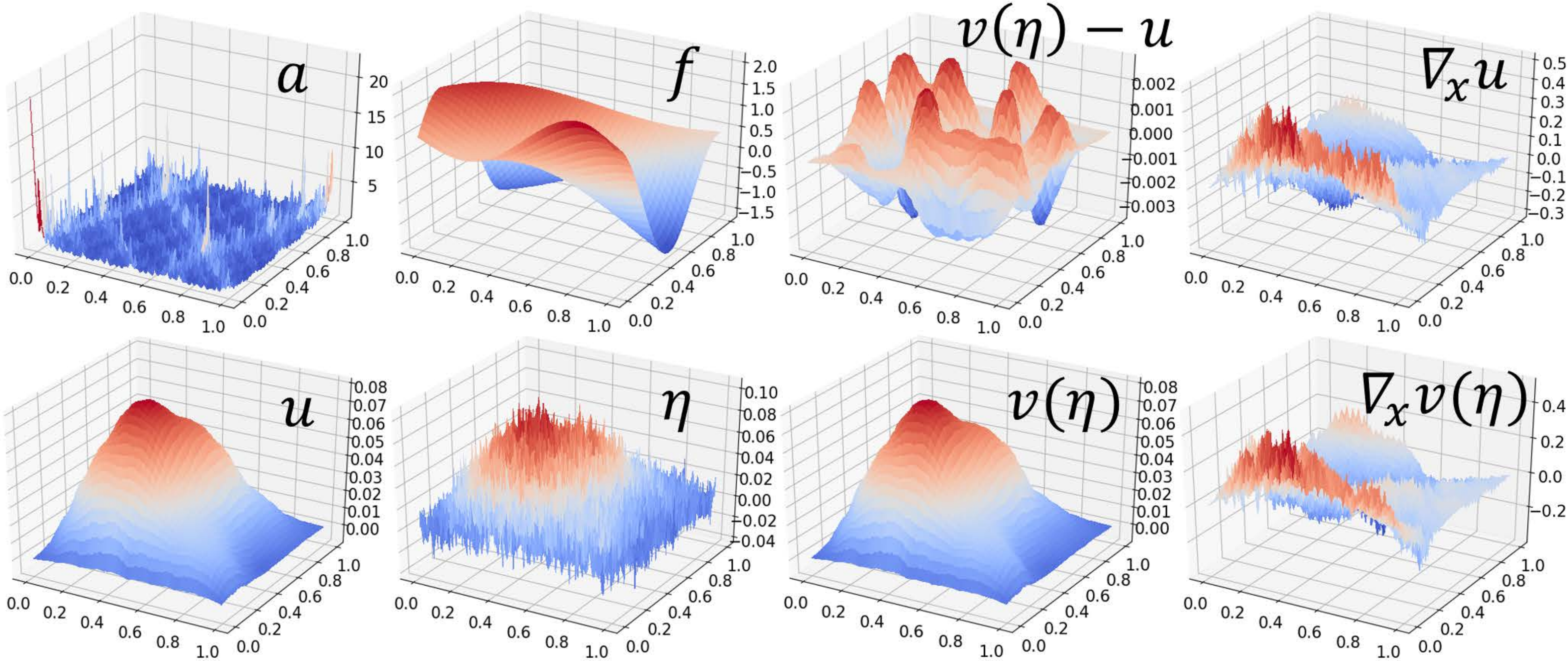}
    \caption{The plots of $a$, smooth $f$, $u$, $\eta$, the near minimax recovery $v(\eta)=\eta^{(l^\dagger)}$, its error from $u$, and the gradient of $u$ and $v(\eta)$}
    \label{plot minimax smooth a, f, u, eta 2D}
\end{figure}
\subsubsection{Example \ref{egprotoa2} with $d=2$} \label{minimax plots 2D}
Consider Example \ref{egprotoa2} with $d=2$.  Take $\Omega = [0, 1]^2$ and $q = 7$.  Use the pre-Haar wavelets  defined as $\phi^{(k)}_{i,j} = 1_{[\frac{i - 1}{2^k}, \frac{i}{2^k}] \times [\frac{j - 1}{2^k}, \frac{j}{2^k}]}$ for $1 \leq i, j \leq 2^k$.
Let $W^{(k)}$  be defined  be the $3(4^{k-1})$ by $4^k$ matrix defined as in construction 4.13 of \cite{OwhadiMultigrid:2015}.

Let $\L = -\diiv (a  \nabla \cdot)$ with
\begin{equation}
    a(x, y) := \prod_{k = 1}^{7} (1 + 0.25 \cos(2^k \pi (x + y))(1 + 0.25 \cos(2^k \pi (x - 3y)).
\end{equation}

Select $f(x)$ at random using the uniform distribution on the unit $L^2(\Omega)$-sphere of $\Phi^{(q)}$.  Average errors are computed using  $100$  randomly (and independently) generated $f$.  Let $\zeta$  be white noise (as in \eqref{eqjkjhbejhbdd}) with $\sigma = 0.001$ and $\eta = u + \zeta$. See Fig.~\ref{plot minimax a, f, u, eta 2D} for numerical illustrations with a particular realization of $f$.

Next consider a case where $f$ is smooth, i.e. $f(x, y) = \cos(3x + y) + \sin(3y) + \sin(7x - 5y)$ . Let  $\zeta$ be white noise with standard deviation $\sigma = 0.01$.  See \ref{plot minimax smooth a, f, u, eta 2D} for the corresponding numerical illustrations.

\section{Comparisons}
We will now compare \eqref{eqkjehgdhdgfh} with soft-thresholding and regularization.

\subsection{Hard and soft thresholding} \label{gamblet thresholding algo}
We call \emph{hard-thresholding} the recovery of $u$ with
\begin{equation}\label{eqhardtreshold}
v(\eta)=\sum_{k=1}^q  \sum_{i\in \J^{(k)}} H^{t^{(k)}}([\phi_i^{(k),\chi},\eta]) \,  \chi_i^{(k)}
\end{equation}
and
\begin{equation} \label{hard var threshold}
    H^{\beta}(x) = \begin{cases}
    x & |x| > \beta\\
    0 & |x| \leq \beta \, .
    \end{cases}
\end{equation}
We call \emph{soft-thresholding} the recovery of $u$ with
\begin{equation}
v(\eta)=\sum_{k=1}^q  \sum_{i\in \J^{(k)}} S^{t^{(k)}}([\phi_i^{(k),\chi},\eta]) \,  \chi_i^{(k)}
\end{equation}
and
\begin{equation} \label{soft var threshold}
    S^{\beta}(x) = \begin{cases}
    x - \beta \, \text{sgn}(x) & |x| > \beta\\
    0 & |x| \leq \beta \, .
    \end{cases}
\end{equation}

The parameters $(t_1,\ldots,t_q)$ are adjusted to achieve minimal average errors.
Since the mass matrix of $\phi^{\chi}$ is comparable to identity (see Thm.~\ref{phi^chi bounds}),
decomposing $f$ over $\phi^\chi$ and the bi-orthogonality identities $[\phi^\chi_i,\chi_j]=\delta_{i,j}$ imply that
$[f,\chi]$ is approximately uniformly sampled on the unit sphere of $\R^{\J}$ and the variance of
$[f,\chi_i^{(k)}]$ can be approximated by $1/|\J|$. Therefore
$[\phi^\chi,u]= B^{(k),-1} [f,\chi^{(k)}]$ and
\eqref{eqcond} imply that the standard deviation of $[\phi^{(k),\chi},u]$ can be approximated by $h^{-2 ks}/\sqrt{|\J|}$.
Therefore optimal choices follow the power law $t_k = h^{-2k s} t_0$ for some parameter $t_0$.

\subsection{Regularization}\label{regularized L^2 norm min}
We call \emph{regularization} the recovery of $u$ with $v(\eta)$ defined as the minimizer of
\begin{equation}
\|v(\eta) - \eta\|_{L^2(\Omega)}^2 + \alpha\|v(\eta)\|^2\,.
\end{equation}
For practical implementation we consider
 $A=\<\tilde{\psi}_i,\tilde{\psi}_j\>$,  the $N\times N$ stiffness matrix obtained by discretizing $\L$ with finite elements
$\tilde{\psi}_1,\ldots,\tilde{\psi}_N$, and write
$\eta=\sum_{i=1}^N y_i \tilde{\psi}_i$ and $\zeta=\sum_{i=1}^N z_i \tilde{\psi}_i$
for the  representation of $\eta$ and $\zeta$ over this basis ($\eta=u+\zeta$ and $z\sim \mathcal{N}(0,\sigma^2 I_d) $, writing $I_d$ for the identity matrix).
In that discrete setting we have
\begin{equation}
v(\eta)=\sum_{i=1}^N x_i \tilde{\psi}_i
\end{equation}
where $x$ is the minimizer of
\begin{equation}\label{eqkejdhidhdh}
|x-y|^2 + \alpha x^T A x
\end{equation}
Theorem \ref{operator min alg} and Corollary \ref{cordjhdghdg} show that
this recovery corresponds to minimizing the energy norm  $\|v\|^2=x^T A x$  subject to $|x-y|\leq \gamma$ with
\begin{equation}\label{eqjwhdgdgyuegduydg}
\gamma=|(I-(\alpha A+ I)^{-1})y|
\end{equation}
In practice $\gamma$ would
correspond to a level of confidence (e.g. chosen so that $\mathbb{P}[|z| > \gamma] = 0.05$ with $z\sim \mathcal{N}(0,\sigma^2 I_d) $).
\begin{Theorem}\label{operator min alg}
Let $x$ be the minimizer of
\begin{equation}\label{eqkjehgdhgjdhg}
\begin{cases}
\text{Minimize }&x^T A x\\
\text{subject to }&|x-y|\leq \gamma\,.
\end{cases}
\end{equation}
If $|y|\leq \gamma$, then $x = 0$.  Otherwise (if $|y|>\gamma$), then $x = (\alpha A + I)^{-1}y$ where $\alpha$ is defined as the solution of \eqref{eqjwhdgdgyuegduydg}.
\end{Theorem}
\begin{proof}
    Supposing $|y|\leq \gamma$, then if $x = 0$, then $|x - y| \leq \gamma$.  Further, $x = 0$ is the global minimum of $x^T A x$.  Therefore in this case, $x = 0$.

    If $|y|>\gamma$ then at minimum $x$ the hyperplane tangent to the ellipsoid of center zero must also be tangent to the sphere of center $y$ which implies that $Ax= \alpha^{-1} (y-x)$ for some parameter $\alpha$. We therefore have $x=(\alpha A+ I)^{-1} y$ and $\alpha$ is determined by the equation $|x-y|=\gamma$ which leads to
    \begin{equation}\label{eqkjekdhdkhdus}
    |(I-(\alpha A+ I)^{-1})y|=\gamma\,.
    \end{equation}
\end{proof}

\begin{Corollary}\label{cordjhdghdg}
    If $|y| > \gamma$, then the minimizers of \eqref{eqkjehgdhgjdhg} and \eqref{eqkejdhidhdh} are identical with $\alpha$ identified as in \eqref{eqkjekdhdkhdus}.
\end{Corollary}
\begin{proof}
     $\nabla_x(|x-y|^2+\alpha x^T A x)=0$ is equivalent to $x-y+ \alpha Ax=0$ which leads to $x=(\alpha A+ I)^{-1} y$.
\end{proof}

\subsection{Numerical experiments}
\subsubsection{Example \ref{egprotoa2} with $d=1$} \label{comp of methods 1D}
Consider the same example as in Subsection \ref{minimax plots 1D}.
The following table shows errors measured in $L^2$ and energy norms averaged over $3,000$ independent random realizations of $f$  and $\zeta$ ($f$ is uniformly distributed over the unit sphere of $L^2(\Omega)$ and $\zeta$ is white noise with $\sigma = 0.001$).  The hard variable thresholding recovery is as defined in section \ref{gamblet thresholding algo}, regularization recovery is as defined in section \ref{regularized L^2 norm min}, and the near minimax recovery refers to $v^\dagger(\eta)=\eta^{(l^\dagger)}$ in Theorem \ref{main thm}.
\begin{center}
\begin{tabular}{ |p{4.25cm}||p{1.768cm}|p{1.768cm}||p{1.768cm}|p{1.768cm}|  }
 \hline
 \multicolumn{5}{|c|}{Comparison of de-noising algorithms performance} \\
 \hline
 Algorithm & $\L$ Error AVG & $\L$ Error STDEV & $L^2$ Error AVG & $L^2$ Error STDEV \\
 \hline
 Hard variable threshold & $4.78 \times 10^{-3}$ & $9.64 \times 10^{-4}$ & $2.25 \times 10^{-4}$ & $1.07 \times 10^{-4}$ \\
 Soft variable threshold & $4.27 \times 10^{-3}$ & $7.70 \times 10^{-4}$ & $1.65 \times 10^{-4}$ & $5.63 \times 10^{-5}$\\
 Regularization recovery & $4.37 \times 10^{-3}$ & $7.93 \times 10^{-4}$ & $2.82 \times 10^{-4}$ & $7.83 \times 10^{-5}$ \\
 Near minimax recovery & $3.90 \times 10^{-3}$ & $5.30 \times 10^{-4}$ & $1.24 \times 10^{-4}$ & $2.50 \times 10^{-5}$ \\
 \hline
\end{tabular}
\end{center}

For reference, the average and standard deviation of the (discrete) energy norm of  $\zeta$ used in this trial were $1.68$ and $0.06$ respectively.

\subsubsection{Example \ref{egprotoa2} with $d=2$} \label{comp of methods 2D}
Consider the same example as in Subsection \ref{minimax plots 2D}.
The following table shows  errors measured in $L^2$ and energy norms averaged over $100$ independent random realizations of $f$  and $\zeta$ ($f$ is uniformly distributed over the unit sphere of $L^2(\Omega)$ and $\zeta$ is white noise with $\sigma = 0.001$).

\begin{center}
\begin{tabular}{ |p{4.25cm}||p{1.768cm}|p{1.768cm}||p{1.768cm}|p{1.768cm}|  }
 \hline
 \multicolumn{5}{|c|}{Comparison of de-noising algorithms performance} \\
 \hline
 Algorithm & $\L$ Error AVG & $\L$ Error STDEV & $L^2$ Error AVG & $L^2$ Error STDEV \\
 \hline
 Hard variable threshold & $6.95 \times 10^{-3}$ & $9.78 \times 10^{-5}$ & $1.42 \times 10^{-4}$ & $7.76 \times 10^{-6}$\\
 Soft variable threshold & $7.18 \times 10^{-3}$ & $1.57 \times 10^{-4}$ & $1.90 \times 10^{-4}$ & $2.35 \times 10^{-5}$\\
 Regularization recovery & $6.90 \times 10^{-3}$ & $1.03 \times 10^{-4}$ & $1.86 \times 10^{-4}$ & $1.88 \times 10^{-5}$\\
 Near minimax recovery & $6.94 \times 10^{-3}$ & $9.58 \times 10^{-5}$ & $1.40 \times 10^{-4}$ & $7.29 \times 10^{-6}$\\
 \hline
\end{tabular}
\end{center}

For reference, the average and standard deviation of the (discrete) $\L$ norm of this trial's $\zeta$ were $0.250$ and $0.06$ respectively.

\section{Proof of near minimax recovery in energy norm}\label{proof of minimax}
This section will provide a proof of Theorem \ref{main thm}. Throughout this section we will use the pre-Haar wavelets of Subsection \ref{pre-haar} as measurement functions $\phi_i^{(k)}$.

\subsection{Bounds on the covariance matrix of the noise in the gamblet basis}
For $s<k$, write $\pi^{(s,k)}:=\pi^{(s,s+1)}\pi^{(s+1,s+2)}\cdots \pi^{(k-1,k)}$ and let $\pi^{(k,s)}:=(\pi^{(s,k)})^T$.
Let $\pi^{(k,k)}:=I^{(k)}$. Let $Z$ be the $\J\times \J$ matrix defined by
\begin{equation}
Z_{i, j} = (N^{(s),T} \pi^{(s, k)} N^{(k)})_{i, j} \text{ for }i \in \J^{(s)} \text{ and }j \in \J^{(k)}\,.
\end{equation}

Observe that by linearity
\begin{equation}
    [\phi^{\chi}, \eta] = [\phi^{\chi}, u] + [\phi^\chi, \zeta]\,,
\end{equation}
where $[\phi^\chi, \zeta]$ is a centered Gaussian vector whose covariance matrix  is $\sigma^2 Z$ as shown in the following lemma.

\begin{Lemma}\label{Noise Cov}
It holds true that
\begin{equation}
\E\big[[\phi^\chi,\zeta] [\phi^\chi,\zeta]^T  \big]=\sigma^2 Z.
\end{equation}
Furthermore, for $x\in \R^\J$,
\begin{equation}\label{eqlkjhkjhjhhuih}
x^T Z x=\|x\cdot \phi^\chi\|_{L^2(\Omega)}^2\,.
\end{equation}
\end{Lemma}
\begin{proof}
  $\E\big[[\phi_i^{(q)},\zeta][\phi_j^{(q)},\zeta]\big]=\sigma^2 \<\phi_i^{(q)},\phi_j^{(q)}\>_{L^2(\Omega)}=\sigma^2 \delta_{i,j}$ implies
\begin{equation}
\E\Big[[\phi_i^{(k),\chi},\zeta] [\phi_j^{(k),\chi},\zeta] \Big]=(N^{(k),T}N^{(k)})_{i,j} \sigma^2
\end{equation}
Therefore, $\phi^{(s)}=\pi^{(s,k)}\phi^{(k)}$ for $1\leq s<k \leq q$ implies
\begin{equation}
\E\Big[[\phi_i^{(s),\chi},\zeta] [\phi_j^{(k),\chi},\zeta] \Big]=(N^{(s),T}\pi^{(s,k)}N^{(k)})_{i,j} \sigma^2\,,
\end{equation}
for $i\in \J^{(s)}$ and $j\in \J^{(k)}$.  The proof of \eqref{eqlkjhkjhjhhuih} is identical.
\end{proof}

There exists a large literature comparing thresholding and component filtering with minimax signal recovery.  Most rigorous results make the assumption that the noise in the decomposition (i.e. $\epsilon z_I$ in $y_I = \theta_I + \epsilon z_I$ as in  \cite{Donoho1992soft, Donoho1990minimax}) is i.i.d. $\N(0, \sigma^2)$.  Although the situation in the gamblet decomposition is slightly different (since the entries of
$[\phi^\chi, \zeta]$ may be correlated and non-identically distributed), the uniform bound  $I_d \leq Z \leq (1+C) I_d$ obtained in
Theorem \ref{phi^chi bounds} will be sufficient to prove near minimax recovery in energy norm.
To obtain Theorem \ref{phi^chi bounds} we will first need the following lemma.
\begin{Lemma}\label{lemdjhegd3uyd}
It holds true that for $k\in \{2,\ldots,q\}$, $z^{(k)}\in \R^{\J^{(k)}}$ and $y^{(k-1)}\in \R^{\I^{(k-1)}}$
\begin{equation}
|z^{(k)}|^2+|y^{(k-1)}|^2 \leq |\pi^{(k,k-1)}y^{(k-1)}+ N^{(k)} z^{(k)}|^2\leq (1+C) |z^{(k)}|^2+|y^{(k-1)}|^2
\end{equation}
where $C$ depends only on $\|\L\|, \|\L^{-1}\|, \Omega, s, d$ and $\delta$.
\end{Lemma}
\begin{proof}
We start with the argument of the proof of \cite[Lem.~14.2]{OwhScobook2018}.

Since $\Img(W^{(k),T})$ and $\Img(\pi^{(k,k-1)})$ are orthogonal and
\begin{equation}
\dim(\R^{\I^{(k)}})=\dim\big(\Img(W^{(k),T})\big)+\dim\big(\Img(\pi^{(k,k-1)})\big)\,,
\end{equation}
for $x\in \R^{\I^{(k)}}$ there exists a unique $z\in \R^{\J^{(k)}}$ and $y\in \R^{\I^{(k-1)}}$ such that
\begin{equation}
x=W^{(k),T}z+\pi^{(k,k-1)} y\,.
\end{equation}
$W^{(k)}\pi^{(k,k-1)}=0$ and $W^{(k)}W^{(k),T}=J^{(k)}$ implies
 $W^{(k)} x=z$.   $R^{(k-1,k)}\pi^{(k,k-1)}=I^{(k-1)}$ implies $R^{(k-1,k)}x =R^{(k-1,k)} W^{(k),T}z+y$.
 Writing
 \begin{equation}
 P^{(k)}=\pi^{(k,k-1)}R^{(k-1,k)}\,,
 \end{equation}
 observe that $P^{(k)}$ is a projection (since $ (P^{(k)})^2= P^{(k)}$) and
\begin{equation}\label{eqjhegdjhgjd}
x= W^{(k),T} W^{(k)}x+ P^{(k)}(I^{(k)}- W^{(k),T}W^{(k)})x \,.
\end{equation}
Observe that $N^{(k)}=A^{(k)}W^{(k),T}B^{(k),-1}$ implies that $W^{(k)} N^{(k)}=J^{(k)}$ and
$P^{(k)}N^{(k)}=\pi^{(k,k-1)} A^{(k-1)}\pi^{(k-1,k)}\Theta^{(k)}N^{(k)}=0$.
Therefore, taking  $x=N^{(k)} z^{(k)}$ in \eqref{eqjhegdjhgjd} implies
\begin{equation}
N^{(k)} z^{(k)}= W^{(k),T} z^{(k)}- P^{(k)} W^{(k),T}z^{(k)} \,.
\end{equation}
Using $P^{(k)} \pi^{(k,k-1)}=\pi^{(k,k-1)}$ we deduce that
\begin{equation}
\pi^{(k,k-1)}y^{(k-1)}+N^{(k)} z^{(k)}= W^{(k),T} z^{(k)}+ P^{(k)} (\pi^{(k,k-1)}y^{(k-1)}-W^{(k),T}z^{(k)}) \,.
\end{equation}
Using $W^{(k)}P^{(k)}=0$ we deduce that
\begin{equation}
|\pi^{(k,k-1)}y^{(k-1)}+N^{(k)} z^{(k)}|^2= |z^{(k)}|^2+ |P^{(k)} (\pi^{(k,k-1)}y^{(k-1)}-W^{(k),T}z^{(k)})|^2 \,.
\end{equation}
Using $P^{(k)} \pi^{(k,k-1)}=\pi^{(k,k-1)}$, $W^{(k)}P^{(k)}=0$  and $W^{(k)}\pi^{(k,k-1)}=0$
we have
\begin{equation}
|P^{(k)} (\pi^{(k,k-1)}y^{(k-1)}-W^{(k),T}z^{(k)})|^2=|y^{(k-1)}|^2+ |P^{(k)} W^{(k),T}z^{(k)}|^2\,.
\end{equation}
Summarizing we have obtained that
\begin{equation}
|\pi^{(k,k-1)}y^{(k-1)}+N^{(k)} z^{(k)}|^2= |z^{(k)}|^2+ |y^{(k-1)}|^2+ |P^{(k)} W^{(k),T}z^{(k)}|^2 \,.
\end{equation}
\cite[Lem.~14.3]{OwhScobook2018} and \cite[Lem.~14.46]{OwhScobook2018} imply that $\|P^{(k)}\|_2 \leq C$ which concludes the proof.
\end{proof}

\begin{Theorem} \label{phi^chi bounds}
It holds true that for $z\in \R^\J$,
\begin{equation}\label{eqkjhedkjdhdh}
|z|^2 \leq \|z\cdot \phi^\chi\|_{L^2(\Omega)}^2 \leq (1+C) |z|^2
\end{equation}
where $C$ depends only on $\|\L\|, \|\L^{-1}\|, \Omega, s, d$ and $\delta$. In particular,
\begin{equation}\label{eqkjhedhhjhjkjdhdh}
I_d \leq Z \leq (1+C) I_d\,.
\end{equation}
\end{Theorem}
\begin{proof}
\eqref{eqkjhedhhjhjkjdhdh} follows from \eqref{eqlkjhkjhjhhuih} and \eqref{eqkjhedkjdhdh}. To prove \eqref{eqkjhedkjdhdh} write
 $z=(z^{(1)},\ldots,z^{(q)})$ with $z^{(k)}\in \R^{\J^{(k)}}$.
Observe that
\begin{equation}
\|z\cdot \phi^\chi\|_{L^2(\Omega)}^2=|\pi^{(q,1)}z^{(1)}+ \sum_{k=2}^q \pi^{(q,k)}N^{(k)} z^{(k)}|^2
\end{equation}
Therefore
\begin{equation}
\|z\cdot \phi^\chi\|_{L^2(\Omega)}^2=|\pi^{(q,q-1)}y^{(q-1)}+ N^{(q)} z^{(q)}|^2
\end{equation}
with
\begin{equation}
y^{(q-1)}=\pi^{(q-1,1)}z^{(1)}+ \sum_{k=2}^{q-1} \pi^{(q-1,k)}N^{(k)} z^{(k)}
\end{equation}
Using Lemma \ref{lemdjhegd3uyd} with $k=q$ implies that
\begin{equation}
|z^{(q)}|^2+|y^{(q-1)}|^2 \leq |\pi^{(q,q-1)}y^{(q-1)}+ N^{(q)} z^{(q)}|^2\leq (1+C) |z^{(q)}|^2+|y^{(q-1)}|^2
\end{equation}
We  conclude the proof of \eqref{eqkjhedkjdhdh} by a simple induction using Lemma \ref{lemdjhegd3uyd} iteratively. For the sake of clarity we will write the next step of this iteration.
We have
\begin{equation}
y^{(q-1)}=\pi^{(q-1,q-2)}y^{(q-2)}+ N^{(q-1)}z^{(q-1)}
\end{equation}
with
\begin{equation}
y^{(q-2)}=\pi^{(q-2,1)}z^{(1)}+ \sum_{k=2}^{q-2} \pi^{(q-2,k)}N^{(k)} z^{(k)}
\end{equation}
and Lemma \ref{lemdjhegd3uyd} with $k=q-1$ implies
\begin{equation}
|z^{(q-1)}|^2+|y^{(q-2)}|^2 \leq |y^{(q-1)}|^2\leq (1+C) |z^{(q-1)}|^2+|y^{(q-2)}|^2\,.
\end{equation}
\end{proof}

\subsection{Near minimax recovery results}

For $T\in \{0,1\}^\J$, write
\begin{equation}\label{eqjkjkjhjkjj}
v_T(\eta):=\sum_{i\in \J} T_i [\phi_i^\chi,\eta]\chi_i\,.
\end{equation}
Let $M>0$ and write  $V_M^{(q)} := \{ u \in \Psi^{(q)}\mid \|\L u\|_{L^2(\Omega)} \leq M\}$.
Define $T^\dagger \in \{0,1\}^\J$ with $T^\dagger_{i} = 1$ if and only if $i \in \J^{(k)}$ for $k \leq l^\dagger$ with
\begin{equation}
    l^\dagger = \argmin_{l\in \{0,\dots,q\}} \beta_l,
\end{equation}
for
\begin{equation}
    \beta_l = \begin{cases}
    h^{2s} M^2 & \text{if } l = 0\\
    \sigma^2 h^{-(2s+d)l} + h^{2s(l+1)}M^2 & \text{if } 1 \leq l \leq q - 1\\
    h^{-(2s+d)q} \sigma^2 & \text{if } l = q \, .
    \end{cases}
\end{equation}
We will first prove the following theorem showing the near optimality of $v_{T^\dagger}$ in the class of estimators of the form \eqref{eqjkjkjhjkjj}.

\begin{Theorem} \label{near minimax recov}
It holds true that
    \begin{equation}
        \sup_{u\in V_M^{(q)}} \E\big[\|u - v_{T^\dagger}(\eta)\|^2 \big] < C \inf_{T\in \{0,1\}^\J} \sup_{u\in V_M^{(q)}} \E\big[\|u - v_{T}(\eta)\|^2 \big]\,,
    \end{equation}
      where $C$ depends only on $h$, $s$, $\|\L\|$, $\|\L^{-1}\|$, $\Omega$, $d$, and $\delta$.
\end{Theorem}
\begin{proof}
To simplify notations, we will write $C$ for any constant depending only  on $h$, $s$, $\|\L\|$, $\|\L^{-1}\|$, $\Omega$, $d$, and $\delta$ (therefore $C h^{-2s}(1+C)$ will still be written $C$).
Abusing notations and writing $T$ for the $\J\times \J$ diagonal matrix with entries $T_{i,i}=T_i$, we have for $u\in \Psi^{(q)}$,
\begin{equation} \label{de-noising error L norm}
    \|u - v_T(\eta)\|^2 = ((I-T) [\phi^\chi, u] - T [\phi^\chi, \zeta])^T B ((I-T) [\phi^\chi, u] - T [\phi^\chi, \zeta]).
\end{equation}
Therefore the bounds \eqref{eqcond} on the (diagonal) blocks of $B$ imply that

\begin{equation}\label{minimax filter1}
    \inf_T \sup_u \E\big[\|u - v_{T}(\eta)\|^2 \big] \geq C^{-1}  \inf_T \sup_u  \sum_{k=1}^q h^{-2sk} \sum_{i\in \J^{(k)}} \bigg( (1-T_i)[\phi_i^{(k),\chi}, u]^2 + T_i \E \big[ [\phi_i^{(k),\chi}, \zeta]^2 \big] \bigg)
\end{equation}
Theorem \ref{phi^chi bounds} implies
 \begin{equation}\label{eqjhkjhkjhgyy}
 \sigma^2 \leq \E \big[ [\phi_i^{(k),\chi}, \zeta]^2 \big] \leq (1+C)\sigma^2\,,
 \end{equation} and
\begin{equation} \label{minimax filter2}
 \inf_T \sup_u \E\big[\|u - v_{T}(\eta)\|^2 \big] \geq C^{-1}  \inf_T \sup_u \sum_k h^{-2sk} \sum_{i\in \J^{(k)}} \bigg( (1-T_i) [\phi_i^{(k),\chi}, u]^2 + T_i \sigma^2 \bigg)\,.
\end{equation}
Next, if $|\cdot|_2^2$ is the $l^2$ norm on $\R^\J$, $\L u = [\phi^\chi, u] \cdot B \phi^\chi$ and Theorem \ref{phi^chi bounds} imply
\begin{equation}
|[\phi^\chi, u]^T B|_2^2 \leq  \|\L u\|_{L^2}^2 \leq (1+C) |[\phi^\chi, u]^T B|_2^2\,.
\end{equation}
We deduce that $\ubar{V}_M^{(q)} \subset V_M^{(q)}\subset \bar{V}_M^{(q)}$ with
\begin{equation}
   \ubar{V}_M^{(q)} :=\{u\in \Psi^{(q)} \mid \sum_{k=1}^q h^{-4ks}|[\phi^{(k),\chi}, u]|^2 <  \frac{M^2}{C} \} \,
\end{equation}
and
\begin{equation}\label{eqkjkjjhgjgj}
\bar{V}_M^{(q)} :=\{u\in \Psi^{(q)} \mid \sum_{k=1}^q h^{-4ks}|[\phi^{(k),\chi}, u]|^2 < C M^2 \}\,.
\end{equation}
We deduce that
\begin{equation} \label{minimax filter3}
  \inf_T \sup_{u\in V_M^{(q)}} \E\big[\|u - v_{T}(\eta)\|^2 \big]  \geq C^{-1}\inf_T \sup_{u \in \ubar{V}_M^{(q)}}  \sum_{k=1}^q h^{-2sk} \sum_{i\in \J^{(k)}} \bigg( (1-T_i)[\phi_i^{(k),\chi}, u]^2 + T_i\sigma^2 \bigg)\,.
\end{equation}
Let us now show that for some $0 \leq l \leq q$, the minimizer $T^{(l),*}$ of the right-hand side of \eqref{minimax filter3} satisfies $T^{(l),*}_i = 1$ for all $i \in \J^{(k)}$ with $k \leq l$ and $T^{(l),*}_i = 0$ otherwise.  For $l\in \{0,\ldots,q-1\}$ write
\begin{equation}
    F^{(l)} := \{T \in \{0,1\}^\J \mid T_i = 1 \text{ for } i \in \J^{(k)} \text{ with } k \leq l, \text{ and } \exists i \in \J^{(l+1)} \text{ s.t. } T_i=0\}\,,
\end{equation}
and write $F^{(q)} := \{ 1\}^\J$. Let us now identify the minimizer in $T$ of
\begin{equation}\label{sup risk}
    \sup_{u \in V_M^{(q)}} \sum_{k=1}^q h^{-2sk}\sum_{i\in \J^{(k)}} \bigg( (1-T_i) [\phi_i^{(k),\chi}, u]^2 + T_i \sigma^2 \bigg).
\end{equation}
Notice that the $u$ which maximizes \eqref{sup risk} can be found by maximizing
\begin{equation} \label{min prob}
\sum_{k=1}^q h^{-2sk}\sum_{i\in \J^{(k)}} (1-T_i) [\phi_i^{(k),\chi}, u]^2
\end{equation}
 over $u\in \Psi^{(q)}$ such that
$\sum_{k=1}^q h^{-4ks}\sum_{i\in \J^{(k)}}  [\phi^{(k),\chi}_i, u]^2 <  \frac{M^2}{C}$. Writing $y_i:=h^{-2ks}[\phi^{(k),\chi}_i, u]$, this is equivalent to maximizing $\sum_{k=1}^q h^{2sk}\sum_{i\in \J^{(k)}} (1-T_i) y_i^2$ over $y\in \R^\J$ such that
$|y|^2\leq \frac{M^2}{C}$.
For $T\in F^{(l)}$ with $1 \leq l \leq q - 1$ the maximum is achieved by taking $y_i=M/\sqrt{C}$ for a single $i\in \J^{(l+1)}$ such that
$T_i=0$ and is equal to $h^{2s(l+1)}M^2/C$. Therefore, using $C \sum_{k=1}^l h^{-2sk} |\J^{(k)}|\geq h^{-(2s+d)l}$ we deduce that
\begin{equation} \label{minimax 4}
    \inf_{T\in F^{(l)}} \sup_{u \in \ubar{V}_M^{(q)}} \sum_{k=1}^q h^{-2sk}\sum_{i\in \J^{(k)}} \bigg( (1-T_i) [\phi_i^{(k),\chi}, u]^2 + T_i \sigma^2 \bigg) \geq \frac{h^{2s(l+1)}M^2 + h^{-(2s + d)l} \sigma^2}{C}.
\end{equation}
For notational convenience, define the sequence $\ubar{\beta}_l$ to be the left-hand side of \eqref{minimax 4} for $0 \leq l \leq q$.  Further, if $l = 0$, then $\ubar{\beta}_0 \geq h^{2s}M^2/C$ with optimum filter $T = \{0\}^\J$ and if $l = q$,  by definition, $T = \{1\}^\J$, hence $\ubar{\beta}_q \geq h^{(2s+d)q}\sigma^2 / C$.
Therefore, \eqref{minimax filter3} implies
\begin{equation}
     \inf_T \sup_{u\in V_M^{(q)}} \E\big[\|u - v_{T}(\eta)\|^2 \big] \geq \min_{l} \ubar{\beta}_l \, . 
\end{equation}
Similarly using \eqref{de-noising error L norm}, \eqref{eqjhkjhkjhgyy} and the bounds \eqref{eqcond} on the (diagonal) blocks of $B$ and \eqref{eqkjkjjhgjgj} imply that
\begin{equation} \label{minideejur3}
   \sup_{u\in V_M^{(q)}} \E\big[\|u - v_{T}(\eta)\|^2 \big]  \leq C  \sup_{u \in \bar{V}_M^{(q)}}  \sum_{k=1}^q h^{-2sk} \sum_{i\in \J^{(k)}} \bigg( (1-T_i)[\phi_i^{(k),\chi}, u]^2 + T_i\sigma^2 \bigg)\,.
\end{equation}
Let $\bar{\beta}_l$ be the value of the right-hand side of \eqref{minideejur3} for $T=T^{(l),*}$.  By the same arguments as those following \eqref{min prob}, we obtain that, $\bar{\beta}_l \leq C(h^{2s(l+1)}M^2 + h^{-(2s + d)l}\sigma^2)$ for $1 \leq l \leq q - 1$, $\bar{\beta}_0 \leq C h^{2s} M^2$, and $\bar{\beta}_q \leq C h^{(2s+d)q}\sigma^2$.  Hence, for some fixed constant, $\bar{\beta}_l \leq C \ubar{\beta}_l$ for all $l$.  Therefore taking $T=T^\dagger$ and bounding the right-hand side of \eqref{minideejur3} as in \eqref{sup risk} leads to
\begin{equation} \label{minideejur3last}
   \sup_{u\in V_M^{(q)}} \E\big[\|u - v_{T^\dagger}(\eta)\|^2 \big] \leq \min_l \bar{\beta}_l \leq C \min_l \ubar{\beta}_l \leq C \inf_T \sup_{u\in V_M^{(q)}} \E\big[\|u - v_{T}(\eta)\|^2 \big] \,,
\end{equation}
which concludes the proof.
\end{proof}

The following theorem shows that gamblet filtering yields a recovery that is near optimal (in energy norm and in the class of all estimators) up to a constant.
\begin{Theorem}
    Suppose that $T \in \{0, 1\}^\J$.  Then the following holds for a constant $C$ dependent on $h$, $s$, $\|\L\|$, $\|\L^{-1}\|$, $\Omega$, $d$, and $\delta$:
    \begin{equation}
        \inf_{T \in \{0, 1\}^\J} \sup_{u \in V_M^{(q)}} \E \big[\|u - v_T(\eta)\|^2\big] \leq C \inf_{v(\eta)} \sup_{u \in V_M^{(q)}} \E \big[\|u - v(\eta)\|^2\big],
    \end{equation}
     where the infimum on the right-hand side is taken over all measurable functions $v: \Psi^{(q)} \rightarrow \Psi^{(q)}$.
\end{Theorem}
\begin{proof}
    Recalling that
    \begin{equation}
        \E \big[\|u - v_T(\eta)\|^2 \big] = \E \big[(T [\phi^\chi, \zeta] - (I-T) [\phi^\chi, u])^T B (T [\phi^\chi, \zeta] - (I-T) [\phi^\chi, u])\big],
    \end{equation}
    the bounds \eqref{eqcond} on the (diagonal) blocks of $B$ imply that
    \begin{equation}\label{filter error}
        \leq C \E \Bigg[ \sum_k \sum_{i\in \J^{(k)}} \bigg( T_i h^{-2sk} (1+C) + (1-T_i) (h^{-sk} [\phi_i^{(k), \chi}, u])^2 \bigg) \Bigg].
    \end{equation}
     We will now introduce preliminary results associated with  the problem of recovering $\theta \in \R^\J$ from the observation of $y = \theta + z$ where $z$ is a centered Gaussian vector with independent entries
     $z_i\sim \N(0, \sigma_i^2)$.  Estimators in this $\R^\J$ space will be written $\hat{\theta}$ (we use $v$ when the problem is formulated in $\Psi^{(q)}$).  For  $T \in \{0, 1\}^\J$ define the estimator  $\hat{\theta}_T$ via
     \begin{equation}\label{eqjkejdheihdudhius}
     (\hat{\theta}_T(y))_i = T_i y_i\,.
      \end{equation}
      Then, writing $|\cdot|_{2}$ for the $l^2$ (Euclidean) norm in $\R^\J$, the expected error of the recovery is:
    \begin{equation}
        \E \big[ |\theta - \hat{\theta}_T(y)|_2^2 \big] = \E \bigg[ \sum_i \Big( T_i z_i^2 + (1-T_i) \theta_i^2 \Big) \bigg].
    \end{equation}
    Let $\bar{\theta}(u)$ be the $\R^\J$ vector defined by  $\bar{\theta}_i(u) = h^{-sk} [\phi_i^{(k), \chi}, u]$ for $i\in \R^\J$.
    Let $\bar{V}_M^{(q)} \supset V_M^{(q)}$ be  defined as in \eqref{eqkjkjjhgjgj} and write $\bar{\theta}(\bar{V}_M^{(q)})$ for the image
    of $\bar{V}_M^{(q)}$ under the map $\bar{\theta}$.
    Consider the problem of recovering $\theta\in \bar{\theta}(\bar{V}_M^{(q)})$ from the observations $y'_i = \theta_i+ z'_i$ where the $z'_i$ are independent centered Gaussian random variables with variance  $\Var(z'_i) = (1+C) h^{-2sk}$ for $i \in \J^{(k)}$.  We will need the following lemma which is directly implied by \cite[Lem.~2]{Minimaxlinear2017}. 
    \begin{Lemma}\label{lemjdhejdhgeuydgse}
    Let $\J$ be an index set and for $\theta\in \R^\J$ let $y$ be the noisy observation of $\theta$ defined by $y=\theta+z$  where $z$ is a Gaussian vector with independent entries
        $z_i\sim \N(0, \sigma_i^2)$. For $a\in \R^\J\setminus\{0\}$ and $M>0$ write
        \begin{equation}
            \Theta_\bold{a}(M) := \Big\{ \theta \in \R^\J \mid \sum_{j \in \J} a_j^2 \theta_j^2 \leq M^2 \Big\},.
        \end{equation}
        Then for $\hat{\theta}_T$ defined as in \eqref{eqjkejdheihdudhius}, it holds  true that
        \begin{equation}\label{eqlkdhjkdjehdk}
            \inf_{T \in \{0,1\}^\J} \sup_{\theta \in \Theta_\bold{a}(M)} \E \big[|\hat{\theta}_T(y) - \theta|_2^2\big] \leq 4.44 \inf_{\hat{\theta}(y)} \sup_{\theta \in \Theta_\bold{a}(M)} \E \big[|\hat{\theta}(y) - \theta|_2^2\big]
        \end{equation}
        where the infimum in the right hand side of \eqref{eqlkdhjkdjehdk} is taken over all functions  $\hat{\theta}$ of $y$.
    \end{Lemma}
    Lemma \ref{lemjdhejdhgeuydgse} implies that
    \begin{equation}
        \inf_T \sup_{V_M^{(q)}} \eqref{filter error} \leq \inf_T \sup_{\bar{V}_M^{(q)}} \eqref{filter error} \leq 4.44 \inf_{\hat{\theta}(y')} \sup_{\theta \in \bar{\theta}(\bar{V}_M^{(q)})} \E \big[|\hat{\theta}(y') - \theta|_2^2\big]\,.
    \end{equation}
    Let $C_0>1$  be a constant larger than $(1+C) h^{-s}$ (with $C$ being the constant in \eqref{eqcond}) and also larger than $\sqrt{1 + C}$ (with $C$ being the constant in \eqref{eqkjhedkjdhdh}).  Then, it is true that $\ubar{\ubar{V}}_M^{(q)} \subset \ubar{V}_M^{(q)} \subset V_M^{(q)} \subset \bar{V}^{(q)}_M$ for
    \begin{equation}
    \ubar{\ubar{V}}_M^{(q)} := \{u\in \Psi^{(q)} \mid C_0 u \in \bar{V}_M^{(q)} \} \,.
    \end{equation}
    Further, for $\theta\in \R^\J$ let $y''=\theta+z''$ be the noisy observation of $\theta$ where $z$ is a centred Gaussian vector with independent entries of variance $\Var(z_i'') = C_0^{-2} \Var(z_i')$. We deduce that
    \begin{equation}\label{prelemma eq}
        \inf_{\hat{\theta}(y')} \sup_{\bar{\theta}(\bar{V}_M^{(q)})} \E \big[|\hat{\theta}(y') - \theta|_2^2\big] = C_0^2 \inf_{\hat{\theta}(y')} \sup_{\bar{\theta}(\bar{V}_M^{(q)})} \E \big[|\hat{\theta}(C_0^{-1} y') - C_0^{-1} \theta|_2^2\big] = C_0^2 \inf_{\hat{\theta}(y'')} \sup_{\bar{\theta}(\ubar{\ubar{V}}_M^{(q)})} \E \big[ |\hat{\theta}(y'') - \theta|_2^2\big] \, .
    \end{equation}
    Define the set of affine recoveries to be of the form $\hat{\theta}_A(y) = Ty+y_0$ with $T$ linear and $y_0 \in \R^\J$.  Since this is a subset of all recoveries $\hat{\theta}$ (and replacing $C_0^2$ with $C$), it holds true that
    \begin{equation}\label{lin y''}
        \eqref{prelemma eq} \leq C \inf_{\hat{\theta}_A(y'')} \sup_{\bar{\theta}(\ubar{\ubar{V}}_M^{(q)})} \E \big[ |\hat{\theta}(y'') - \theta|_2^2\big].
    \end{equation}
    If we define $\theta_i := h^{-sk} [\phi_i^{(k),\chi}, u]$, $z_i = h^{-sk} [\phi_i^{(k), \chi}, \zeta]$, and $y_i = \theta_i + z_i$, it is true that $\Cov(z) > \Cov(z'')$.   Since, in the space of affine recoveries, the recovery error increases with the covariance matrix,  and since $\bar{\theta}(\ubar{\ubar{V}}_M^{(q)}) \subset \bar{\theta}(V_M^{(q)})$, we have
    \begin{equation}\label{aff y}
        \eqref{lin y''} \leq C \inf_{\hat{\theta}_A(y)} \sup_{\bar{\theta}(\ubar{\ubar{V}}_M^{(q)})} \E \big[ |\hat{\theta}(y) - \theta|_2^2\big] \leq C \inf_{\hat{\theta}_A(y)} \sup_{\bar{\theta}(V_M^{(q)})} \E \big[|\hat{\theta}(y) - \theta|_2^2\big]\,.
    \end{equation}
    The following lemma is a direct application of \cite[Cor.~1]{donoho1994statistical} (with $K = I$ and $L = I$, see also  the remarks in \cite[sec.~13.1]{donoho1994statistical}).  It will be used to compare affine and general minimax recovery errors.
    \begin{Lemma}
    Let $\J$ be an index set and $z$ be a centered Gaussian vector of $\R^\J$ with $\Cov(z) = Z$ where $Z$ is invertible.
    Let $V$ be a closed, bounded, and convex subset and $\R^\J$ and for $\theta \in V$ let $y=\theta+z$ be the noisy observation of $\theta$. It holds true that
        \begin{equation}
            \inf_{\hat{\theta}_A(y)} \sup_{\theta\in V} \E \big[|\hat{\theta}(y) - \theta|_2^2\big] \leq 1.25 \inf_{\hat{\theta}(y)} \sup_{\theta\in V} \E \big[|\hat{\theta}(y) - \theta|_2^2\big]\,.
        \end{equation}
    \end{Lemma}
    Since ellipsoids are closed, bounded, and convex, we can apply the lemma and bounds on the (diagonal) blocks of $B$ in \eqref{eqcond} to obtain the desired result,
    \begin{equation}
        \eqref{aff y} \leq C \inf_{\hat{\theta}(y)} \sup_{\bar{\theta}(V_M^{(q)})} \E \big[|\hat{\theta}(y) - \theta|_2^2\big] \leq C h^{-2s} \inf_{v(\eta)} \sup_{u \in V_M^{(q)}} \E \big[\|v(\eta) - u\|^2\big]\,.
    \end{equation}
\end{proof}
\subsection{Proof of Smooth Recovery}\label{smooth recov}
The final statement of Theorem \ref{main thm} will be proven in this section.  First, the following lemma will be proven.
\begin{Lemma} \label{smooth recov lemma}
    Let $T^\dagger$ be as in Theorem \ref{main thm} such that $l^\dagger \neq 0$, then it holds true that
    \begin{equation}
        \|v_{T^\dagger}(\eta)\| \leq \|u\| + C \sqrt{z} \sigma^\frac{2s}{4s+d} M^\frac{2s+d}{4s+d}
    \end{equation}
    with probability at least $1 - (z e^{1-z})^{|\J^{(l^\dagger)}|/2}$ for $C$ a constant dependent on $h$, $s$, $\|\L\|$, $\|\L^{-1}\|$, $\Omega$, $d$, and $\delta$.
\end{Lemma}
\begin{proof}
    Start by observing that since $v_{T^\dagger}(\eta) = u^{(l^\dagger)} + \zeta^{(l^\dagger)}$,
    \begin{equation}
        \|v_{T^\dagger}(\eta)\| \leq \|u^{(l^\dagger)}\| + \|\zeta^{(l^\dagger)}\| \leq \|u\| + \|\zeta^{(l^\dagger)}\|.
    \end{equation}
    Next, the latter part of the sum on the right-hand side will be bounded.
    \begin{equation}\label{noise bound}
        \|\zeta^{(l^\dagger)}\|^2 \leq \sum_{k=1}^{l^\dagger} (1+C) h^{-2sk} \sum_i [\phi^{(k),\chi}_i, \zeta]^2 \leq (1+C) h^{-2sl^\dagger} \sum_{k=1}^{l^\dagger} \sum_i [\phi_i^{(k),\chi}, \zeta]^2.
    \end{equation}
    Defining the Gaussian vector $X$ such that $X_i = [\phi^{(k),\chi}, \zeta]$ for $i \in \J^{(k)}$ for $k = \{0, ..., l^\dagger\}$, it holds that $\Cov(X) < \sigma^2 (1 + C) I$ by Theorem \ref{phi^chi bounds}.  Defining Gaussian vector $\bar{X}$ such that $\Cov(\bar{X}) = \sigma^2 (1 + C) I$, it holds that $M := \sigma \sqrt{1+C} \Cov(X)^{-1/2} > I$.  Note that since $\P[|MX|^2 \geq |X|^2] = 1$ and $MX$ has an identical distribution to $\bar{X}$, it is true that
    \begin{equation}
        \P[|X|^2 < x] \geq \P[|MX|^2 < x] = \P[|\bar{X}|^2 < x] \, .
    \end{equation}
    Hence, to get a tail bound of \eqref{noise bound}, we apply a Chernoff bound on $\chi^2$ distributions, which states for $z > 1$ the CDF of $Q_k$, a $\chi^2(k)$-distribution, is bounded by
    \begin{equation}
        P(Q_k < zk) \geq 1 - (z e^{1-z})^{k/2}.
    \end{equation}
    This bound can be deduced from Lemma 2.2 of \cite{chi2chernoff}.  Hence, with probability at least $1 - (z e^{1-z})^{|\J^{(l^\dagger)}|/2}$ ($C$ absorbs a constant dependent on $h$, $s$, and $d$)
    \begin{equation}\label{lemma 6.7 result}
        \eqref{noise bound} \leq C z \sigma^2 h^{(-2s-d)l^\dagger} \leq Cz \sigma^2 \bigg(\frac{\sigma}{M}\bigg)^\frac{-4s-2d}{4s+d} = Cz \sigma^\frac{4s}{4s+d} M^\frac{4s+2d}{4s+d}
    \end{equation}

\end{proof}

\begin{Theorem}
Assuming $l^\dagger \neq 0$, the following inequality holds with  probability at least $1 - \varepsilon$,
    \begin{equation}
        \|\eta^{(l^\dagger)}\| \leq \|u\| + C \sqrt{\log\frac{1}{\varepsilon}} \sigma^\frac{2s+d}{4s+d} M^\frac{2s+2d}{4s+d}\,,
    \end{equation}
    where $C$ depends only on $h$, $d$, $\delta$ and $s$.
\end{Theorem}
\begin{proof}
Given lemma \ref{smooth recov lemma}  we start by defining
\begin{equation}
    \varepsilon = \big(z e^{1-z}\big)^{|\J^{(l^\dagger)}|/2}.
\end{equation}
We now get an upper bound for $z$ in terms of $\varepsilon$.  Since the only significant values of $\varepsilon$ are $\varepsilon \in (0, 1)$, it must be true that $z > 1$.  This implies that $z e^{1-z} \leq 2 e^{-z/2}$.   Hence,
\begin{equation}
    \frac{4}{\log 2} \log\frac{1}{\varepsilon} \geq |\J^{(l^\dagger)}| z \geq C^{-1} h^{-dl^\dagger}z
\end{equation}
Hence, for $C$ dependent on $h$, $d$, $\delta$ and $s$,
\begin{equation}
    C \log \frac{1}{\varepsilon} \geq \bigg(\frac{\sigma}{M}\bigg)^\frac{-2d}{4s + d} z
\end{equation}
This yields an upper bound on $z$, hence we apply Lemma \ref{smooth recov lemma} and obtain the result:
\begin{equation}
    \eqref{lemma 6.7 result} \leq C \log \frac{1}{\epsilon} \sigma^\frac{4s+2d}{4s+d} M^\frac{4s+4d}{4s+d}.
\end{equation}
\end{proof}

\section{De-noising graph Laplacians} \label{prob graph}
Since gamblets can also be constructed for graph Laplacians \cite{OwhadiScovel:2017,OwhScobook2018} the proposed method can also be used for de-noising
such operators. This section will provide a succinct description of  this extension with numerical illustrations.
\subsection{The problem}
 Let $G=(V,E)$ be a simple graph (with vertex set $V$ and edge set $E$) and let $L$ be its graph Laplacian, which is defined as follows:

\begin{equation}
    L_{i, j} =
    \begin{cases}
        \text{deg}(v_i) & \text{if } i = j \\
        -1 & \text{if } i \neq j \text{ and } v_i \text{ is adjacent to } v_j\\
        0 & \text{otherwise}
    \end{cases}
\end{equation}
Fix $i_0\in V$ and let
 $\V:=\{u\in \R^V \mid u_{i_0}=0\}$. Note that $L$ is a symmetric positive linear bijection from $\V$ to $\V$ with respect to the inner product $\< u,v \>_L = u^T L v$. Consider the following problem.

\begin{Problem}\label{pb3}
 Recover the solution $u\in \V$ of
\begin{equation}\label{eqn:scalar}
L u = f\,
\end{equation}
as accurately as possible in the energy norm, defined as $\|u\|^2_L = \<u,u\>_L$, given the information that $f\in \V$ with $|f|_2\leq M$ and the observation
$\eta=u+\zeta$ where $\zeta$ is a centered Gaussian vector on $\V$ with covariance matrix $\sigma^2$.
\end{Problem}

\subsection{Near minimax recovery}

Similarly to section \ref{Near Minimax Sec}, with $M > 0$, define
\begin{equation}
    V_M = \{ u \in \V \mid |L u|_2 \leq M \} \, .
\end{equation}
With gamblets as defined in section \ref{gamblets R^N}, \cite{OwhadiScovel:2017,OwhScobook2018} provide sufficient conditions under which
the matrices $B^{(k)}$ and $Z$ are uniformly well conditioned  (as in Theorem \ref{themoptimaldecomposition}), in particular under these conditions there exists  $H \in (0,1)$ such that
\begin{equation}
    C^{-1}H^{-2(k-1)}J^{(k)}\leq B^{(k)} \leq CH^{-2k}J^{(k)}\text{ and }\Cond(B^{(k)})\leq CH^{-2}\,.
\end{equation}
In this brief outlook we will simply numerically estimate $H$ and define $d$ to be such that $H^{-d} = \Theta(|\J^{(k)}|)$. Assuming $\sigma > 0$, set
\begin{equation}\label{l^dag defbis}
    l^\dagger = \argmin_{l\in \{0,\dots,q\}} \beta_l,
\end{equation}
for
\begin{equation} \label{beta defbis}
    \beta_l = \begin{cases}
    H^{2} M^2 & \text{if } l = 0\\
    H^{-(2+d)l} \sigma^2 + H^{2(l+1)}M^2 & \text{if } 1 \leq l \leq q - 1\\
    H^{-(2+d)q} \sigma^2 & \text{if } l = q \, .
    \end{cases}
\end{equation}
Finally, we hard thresholding \eqref{eqhardtreshold} and
\begin{equation}\label{eqkedlhdkjehd}
v^\dagger(\eta) := \eta^{(l^\dagger)}\,,
\end{equation}
to recover $u$ from $\eta$.

\subsection{Numerical illustration}\label{sec results graph}
For our numerical illustration  the edges and vertices of the graph are the streets and intersections of Pasadena, CA (obtained using the python OSMNX package \cite{OSMNX}). The latitude and longitude of each vertex is known through the package, and $f$ is defined as $f(v_i) = \cos(3x(v_i) + y(v_i)) + \sin(3y(v_i)) + \sin(7x(v_i) - 5y(v_i))$, where $x(v_i)$ and $y(v_i)$ are the latitudes and longitudes of vertex $v_i$ normalized to $[0,1]$.  The noise, $\zeta$, is a centered Gaussian vector on $\V$ with $\sigma = 400$.

The measurement functions used for the gamblet decomposition are defined as
\begin{equation}
    \phi^{(k)}_{i, j}(u) = \frac{1}{\sqrt{|S^{(k)}_{i,j}|}} \sum_{l \in S^{(k)}_{i,j}} u_{v_l}
\end{equation}
where $S^{(k)}_{i,j}$ contain all vertices with relative latitudes in $[\frac{i - 1}{2^k}, \frac{i}{2^k}]$ and relative longitudes in $[\frac{j - 1}{2^k}, \frac{j}{2^k}]$.
The $W^{(k)}$ matrix are as in Subsection \ref{pre-haar}.  Further, in this graph, we estimate $H \approx 0.5996$ and $d \approx 2.327$.

The following figures show $f$, $u$, $\eta$, $v^\dagger(\eta)$ (using \eqref{eqkedlhdkjehd}), and the recovery error $v^\dagger(\eta) - u$.
\begin{figure}[H]
    \centering
    \includegraphics[width=0.9\textwidth]{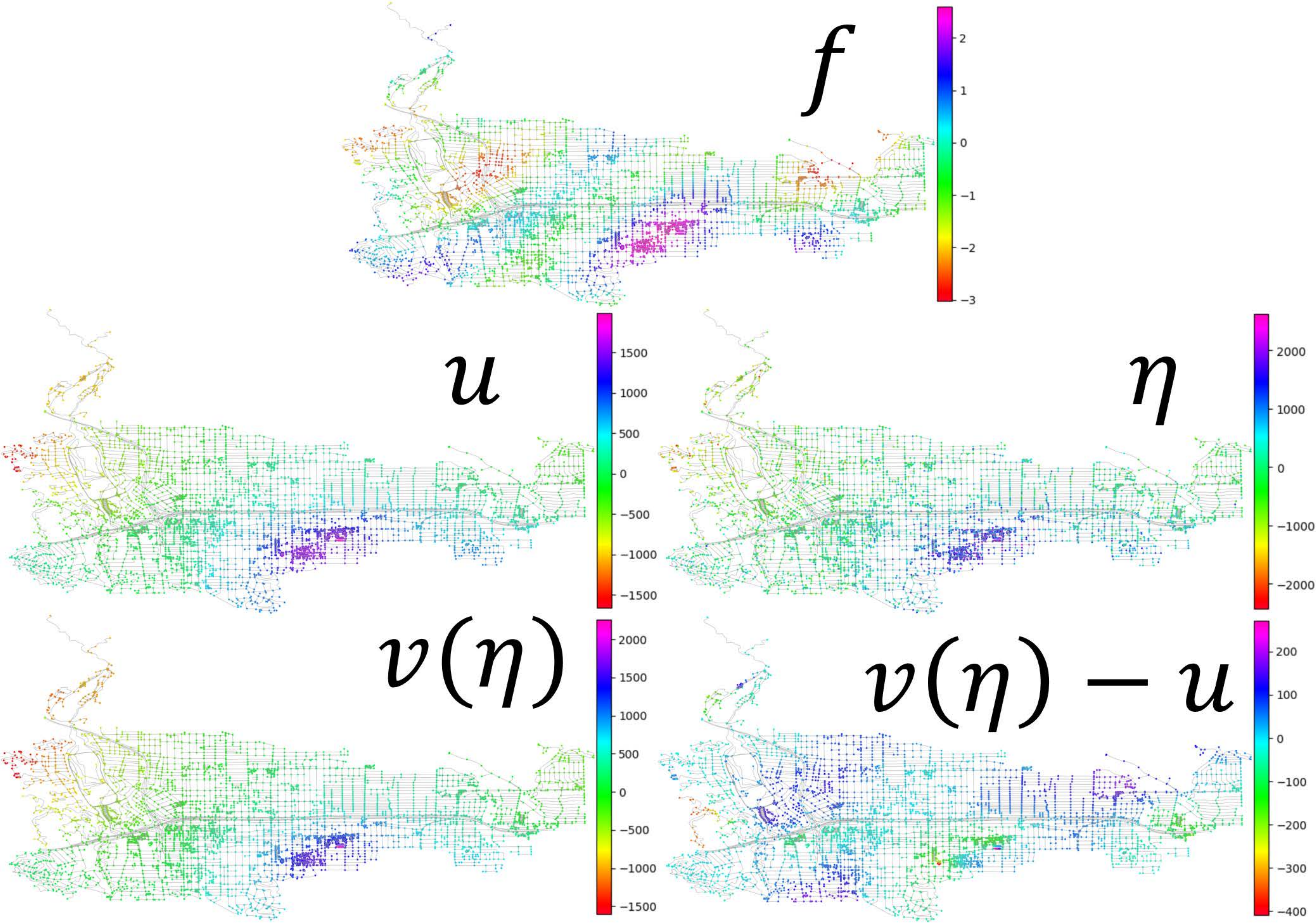}
    \caption{The plot of $f$, $u$, $\eta$, $v^\dagger(\eta)$, and $v^\dagger(\eta)-u$ on the Pasadena graph.}
    \label{f Pasadena}
\end{figure}
Average errors were computed using $100$ independent noise samples.  The de-noising errors of the hard gamblet thresholding with fixed thresholds ($t^{(k)} = t$ for all $k$) is compared to the near minimax recovery analog, $v^\dagger$, is given in the table below.
\begin{center}
\begin{tabular}{ |p{5.25cm}||p{1.768cm}|p{1.768cm}||p{1.768cm}|p{1.768cm}|  }
 \hline
 \multicolumn{5}{|c|}{Comparison of de-noising algorithms performance} \\
 \hline
 Algorithm & $\L$ Error AVG & $\L$ Error STDEV & $L^2$ Error AVG & $L^2$ Error STDEV \\
 \hline
 Estimator \eqref{eqkedlhdkjehd} & $578.3$ & $11.6$ & \num{6216} & $234$\\
 Hard variable threshold & $541.6$ & $40.7$ & \num{5644} & $446$\\

 \hline
\end{tabular}
\end{center}
For reference, the average and standard deviation of the $L$-energy norm of $\zeta$ used in this trial were $49265$ and $576$ respectively.

\paragraph{Acknowledgments.}
The authors gratefully acknowledges this work supported by  the Air Force Office of Scientific Research and the DARPA EQUiPS Program under award   number FA9550-16-1-0054 (Computational Information Games).

\bibliographystyle{plain}
\bibliography{Gambletsfordenoising7}

\end{document}